\documentclass[a4paper,12pt]{amsart}

\usepackage{standalone}
\usepackage{amsmath}
\usepackage{amssymb}
\usepackage{mathrsfs}
\usepackage{ifthen}
\usepackage{graphicx}
\usepackage[T1]{fontenc} 
\usepackage{tikz}
\usetikzlibrary{arrows}
\usetikzlibrary{shapes}
\usetikzlibrary{patterns}
\nonstopmode \numberwithin{equation}{section}
\newtheorem{thm}{Theorem}[section]
\newtheorem{cor}[equation]{Corollary}
\newtheorem{lem}[equation]{Lemma}
\newtheorem{prop}[equation]{Proposition}

\newtheorem{claim}{Claim}

\theoremstyle{definition}
\newtheorem{defn}{Definition}[section]
\newtheorem{example}{Example}[section]
\newtheorem{prob}[equation]{Problem}
\newtheorem{rem}{Remark}[section]
\newtheorem{observation}{Observation}

\newcounter{minutes}
\divide\time by 60
\newcounter{hours}
\multiply\time by 60
\addtocounter{minutes}{-\time}
\newcounter {own}
\def\theown {\thesection       .\arabic{own}}

\newenvironment{pf}[1][]{%
	\vskip 3mm
	\noindent
	\ifthenelse{\equal{#1}{}}%
	{{\slshape Proof. }}%
	{{\slshape #1.} }%
}%
{\qed\bigskip}

\newcounter{alphabet}

\newenvironment{Thm}[1][]{\refstepcounter{alphabet}%
	\bigskip%
	\noindent%
	{\bf Theorem \Alph{alphabet}}%
	\ifthenelse{\equal{#1}{}}{}{ (#1)}%
	{\bf .} \itshape}{\vskip 8pt}
\def\be{\begin{equation}}
	\def\ee{\end{equation}}

\newcommand{\bee}{\begin{enumerate}}
	\newcommand{\eee}{\end{enumerate}}

\newcommand{\blem}{\begin{lem}}
	\newcommand{\elem}{\end{lem}}
\newcommand{\bthm}{\begin{thm}}
	\newcommand{\ethm}{\end{thm}}
\newcommand{\bcor}{\begin{cor}}
	\newcommand{\ecor}{\end{cor}}
\newcommand{\beg}{\begin{examp}}
	\newcommand{\eeg}{\end{examp}}
\newcommand{\begs}{\begin{examples}}
	\newcommand{\eegs}{\end{examples}}
\newcommand{\bdefe}{\begin{defin}}
	\newcommand{\edefe}{\end{defin}}
\newcommand{\bprob}{\begin{prob}}
	\newcommand{\eprob}{\end{prob}}
\newcommand{\bei}{\begin{itemize}}
	\newcommand{\eei}{\end{itemize}}

\begin{document}
	
	\title{Visible quasihyperbolic geodesics}

	\author{Vasudevarao Allu}
	\address{Vasudevarao Allu,
		Discipline of Mathematics,
		School of Basic Sciences,
		Indian Institute of Technology  Bhubaneswar,
		Argul, Bhubaneswar, PIN-752050, Odisha (State),  India.}
	\email{avrao@iitbbs.ac.in}

	\author{Abhishek Pandey}
	\address{Abhishek Pandey,
		Discipline of Mathematics,
		School of Basic Sciences,
		Indian Institute of Technology  Bhubaneswar,
		Argul, Bhubaneswar, PIN-752050, Odisha (State),  India.}
	\email{ap57@iitbbs.ac.in}
	
	\makeatletter
	\@namedef{subjclassname@2020}{\textup{2020} Mathematics Subject Classification}
	\makeatother

	\subjclass[2020]{Primary 30F45, 30L10, 30L99, 30C65. Secondary 51F99, 53C22.}
	\keywords{Gromov hyperbolic spaces, John domains, quasihyperbolic metric, quasihyperbolic geodesics, quasiconformal mappings, visibility}

\def\thefootnote{}
\footnotetext{ {\tiny File:~\jobname.tex,
		printed: \number\year-\number\month-\number\day,
		\thehours.\ifnum\theminutes<10{0}\fi\theminutes }
} \makeatletter\def\thefootnote{\@arabic\c@footnote}\makeatother

\thanks{}

\begin{abstract}
In this paper, motivated by the work of Bonk, Heinonen, and Koskela (Asterisque, 2001), we consider the problem of the equivalence of the Gromov boundary and Euclidean boundary. Our strategy to study this problem comes from the recent work of Bharali and Zimmer (Adv. Math., 2017) and Bracci, Nikolov, and Thomas (Math. Z., 2021). We present the concept of a quaihyperbolic visibility domain (QH-visibility domain) for domains that meet the visibility property in relation to the quasihyperbolic metric. By utilizing this visibility property, we offer a comprehensive solution to this problem. Indeed, we prove that such domains are precisely the QH-visibility domains that have no geodesic loops in the Euclidean closure. Furthermore, we establish a general criterion for a domain to be the QH-visibility domain. Using this criterion, one can determine that uniform domains, John domains, and domains that satisfy quasihyperbolic boundary conditions are QH-visibility domains. We also compare the visibility of hyperbolic and quasihyperbolic metrics for planar hyperbolic domains. As an application of the visibility property, we study the homeomorphic extension of quasiconformal maps. Moreover, we also study the QH-visibility of unbounded domains in $\mathbb{R}^n$. Finally, we present a few examples of QH-visibility domains that are not John domains or QHBC domains.
\end{abstract}

		\maketitle
	\pagestyle{myheadings}
	\markboth{Vasudevarao Allu and Abhishek Pandey}{Visible quasihyperbolic geodesics}
\tableofcontents%\setcounter{tocdepth}{3}

\section{Introduction}

For a bounded domain $\Omega$ in $\mathbb{R}^n$, we consider the Euclidean metric $d_{Euc}$ and the quasihyperbolic metric $k_{\Omega}$, and let $\partial_{Euc}\Omega$ and $\partial_{G}\Omega$ denote the Euclidean boundary and Gromov boundary, respectively. Bonk, Heinonen, and Koskela, in their seminal work \cite{Koskela-2001}, proved the following very important and useful result: uniform domains can be characterized in terms of Gromov hyperbolicity.

\begin{Thm}\cite[Theorem 1.11]{Koskela-2001}
	A bounded domain $\Omega$ in $\mathbb{R}^n$ is uniform if, and only if, $(\Omega,k_{\Omega})$ is Gromov hyperbolic and the identity map $id:(\Omega,k_{\Omega})\to(\Omega,d_{Euc})$ extends as a homeomorphism $\Phi:\overline{\Omega}^G\to \overline{\Omega}^{Euc}$ such that $\Phi$ is quasisymmetrric homeomorphism from $\partial_{G}\Omega$ onto $\partial_{Euc}\Omega$.
\end{Thm}
This motivates us to consider the following problem:

\begin{prob}\label{GB-EB}
	Let $\Omega$ be a bounded domain in $\mathbb{R}^n$ and $(\Omega,k_{\Omega})$ be Gromov hyperbolic. For which such domain the following is true:\\
	The identity map $\mbox{id}:(\Omega,k_D)\to(\Omega,d_{Euc})$ extends to a continuous surjective map or homeomorphism from $\partial_{G}\Omega$ onto $\partial_{Euc}\Omega$?
\end{prob}
V\"ais\"al\"a has established results similar to Theorem A in the Banach space setting \cite{Vaisala-GH-1}. If we follow V\"ais\"al\"a \cite[2.21]{Vaisala-GH-1} terminology, Problem \ref{GB-EB} inquires about the conditions under which the natural map exists. Problem \ref{GB-EB} is an interesting problems in the theory of Gromov hyperbolicity of quasihyperbolic metric. In general, the identity map may not even extend to a continuous map from $\overline{\Omega}^G$ to $\overline{\Omega}^{Euc}$, for instance, take 
\begin{equation}\label{not compact}
	\Omega=(0,1)\times (0,1)\setminus \bigcup_{j=1}^\infty \left(\left\{\frac{1}{2^j}\times \left[0,\frac{1}{2}\right]\right\}\right).
\end{equation}
Even if the identity map extends to a continuous map, the extension may not be injective. For instance, if we take $\Omega=\mathbb{D}\setminus\{(x,y): 0\le x<1, y=0\}$, then then identity map extends to a continuous map from Gromov closure onto the Euclidean closure but the extension is not homeomorphism since one Euclidean boundary point on the radius may define two distinct Gromov boundary points.\\

As per the best of our knowledge, the most general result in this line is due to Lammi \cite{Lammi-2011,Lammi-Thesis}. Lammi \cite{Lammi-2011} has considered Problem \ref{GB-EB} in the metric space that is locally compact, non-complete, and quasi-convex. If we look at Problem \ref{GB-EB} for bounded domains in $\mathbb{R}^n$ with the inner metric, Lammi's work is very helpful. This is because the Gehring-Hayman theorem makes it clear that a domain $\Omega$ in $\mathbb{R}^n$ with the inner metric is always quasi-convex. In this case, it is uncertain whether the inner boundary is compact. For example, if we consider the domain $\Omega$ mentioned in \eqref{not compact} and equip $\Omega$ with the inner metric, then the inner boundary is closed and bounded, but not compact. In this line, we have a topological result by Lammi \cite[Theorem 3.4]{Lammi-Thesis}, which says that the compactness of the inner boundary is equivalent to the homeomorphic extension of the identity map from the Gromov boundary onto the inner boundary. Further, \cite[Theorem 1.1]{Lammi-2011} says that if we have a locally compact, non-complete, quasi-convex space that is Gromov hyperbolic and satisfies the Gehring-Hayman theorem, then all we need is a suitable growth condition of a quasihyperbolic metric for the identity map to extend homeomorphically from the Gromov boundary onto the metric boundary. The proof of Theorem 1.1 in Lammi's paper \cite{Lammi-2011} does not rely on the compactness of the metric boundary, and hence the compactness of the inner boundary is the consequence of this result. For bounded domains in $\mathbb{R}^n$, the Euclidean boundary $\partial_{Euc}\Omega$ is always compact. This means that \cite[Theorem 3.4]{Lammi-Thesis} is relevant for Problem \ref{GB-EB}, and we will mention it below in the Euclidean setting.

\begin{Thm}\cite[Theorem 3.4]{Lammi-Thesis}
	Let $\Omega$ be a bounded domain in $\mathbb{R}^n$. Suppose that $(\Omega,k_{\Omega})$ is Gromov hyperbolic. If $(\Omega,d_{Euc})$ is quasi-convex (see Definition \ref{quasi-convex}), then the identity map $id:(\Omega,k_\Omega)\to(\Omega,d_{Euc})$ extends as a homoemorphism $\Phi:\partial_{G}\Omega\to \partial_{Euc}\Omega$.
\end{Thm}

The goal of this paper is to provide a complete solution to Problem \ref{GB-EB}. Our strategy to study this problem arises from a recent work by Bharali and Zimmer \cite{Bharali-2017}, and Bracci, Nikolov, and Thomas \cite{Bracci-2022}, where they have studied a "suitable form" of visibility in the context of Kobayashi distance. The concept of visibility was first introduced by Eberlein and O'Neill in \cite{O'Neill-1973}, where they basically introduced a general construction of compactification of Hadamrd manifolds by attaching an abstract boundary or boundary at infinity, which is called the ideal boundary. Using the notion of an ideal boundary, Eberlein and O'Neill defined the term visibility manifold for a complete Riemannian manifold with nonpositive sectional curvature. For more details on visibility refer to Section \ref{Appendix}. Bharali and Zimmer \cite{Bharali-2017} have introduced a new class of domains in complex Euclidean space called Goldilocks domains, which satisfy a form of visibility with respect to Kobayashi distance. Using this property, Bharali and Maitra \cite{Bharali-2021} coined the term visibility domain in the context of the Kobayashi metric, and since then it has been systematically studied (see \cite{Bharali-2021, Bharali-2022, Bracci-2022, Bracci-2024, Sarkar-2021}).\\

In addition, visibility domains with respect to Kobayashi distance are accompanied by the property that the identity map  can be extended to a continuous surjective map from the Gromov closure onto the Euclidean closure of a complete hyperbolic bounded domain in $\mathbb{C}^n$. Nikolov, Bracci, and Thomas \cite{Bracci-2022} proved this result in the context of Kobayashi distance. Hence, if we can establish a similar outcome within the framework of the quasihyperbolic metric, it will provide the solution for Problem \ref{GB-EB}. This motivates us to introduce the concept of {\it QH-visibility domain}. Consequently, we provide a complete solution to the Problem \ref{GB-EB} by proving that {\it such domains are precisely QH-visibility domains that have no geodesic loops in $\overline{\Omega}^{Euc}$} (see Theorem \ref{Main-thm-GB-EB} and Theorem \ref{no loop}).\\

Another reason to investigate the visibility property in the context of the quasihyperbolic metric is that visibility, as defined by Eberlein and O'Neill, is a natural replacement of the strict negative sectional curvature of a complete Riemannian manifold (refer to Lemma 9.10 in \cite{O'Neill-1969} and \cite{O'Neill-1973}). Therefore, {\it study of visibility property in the framework of the quasihyperbolic metric may be viewed as a weak notion of negative curvature for domains $\Omega\subset \mathbb{R}^n$ viewed as metric spaces equipped with the quasihyperbolic metric.}\\

The above discussion tells us that  investigating visibility within the framework of quasihyperbolic metric is worthwhile. In light of this, we begin the study of visibility property in the context of the quasihyperbolic metric with the following definition.

\begin{defn}\label{main-defn}
	Let $\Omega$ be a bounded domain in $\mathbb{R}^n$ and $\partial_{Euc}\Omega$ be its Euclidean boundary. Let $p,q\in\partial_{Euc}\Omega$, $p\ne q$. We say that the pair $\{p,q\}$ has {\it visible quasihyperbolic geodesics} if there exist neighborhoods $U,V$ of $p,q$ respectively such that $\overline{U}\cap\overline{V}=\emptyset$ and a compact set $K_{\{p,q\}}\subset \Omega$ such that for any quasihyperbolic geodesic $\gamma:[0,T]\to \Omega$ with $\gamma(0)\in U$ and $\gamma(T)\in V$, we have 
	$$\gamma([0,T])\cap K\ne \emptyset.$$
\end{defn}
\begin{defn}
	We say that a bounded domain $\Omega$ in $\mathbb{R}^n$ has visibility property if any pair of distinct points $\{p,q\}\subset \partial_{Euc}\Omega$ has visible quasihyperbolic geodesics.
\end{defn}
\begin{defn}
	A bounded domain $\Omega$ in $\mathbb{R}^n$ is said to be quasihyperbolic visibility domain, in short, {\it QH-visibility domain} if it has visibility property.
\end{defn}

\begin{rem}
	Note that the definition of QH-visibility domain is equivalent to the following.
	Let $\Omega$ be a bounded domain in $\mathbb{R}^n$.  Then $\Omega$ is said to be QH-visibility domain if, and only if, for any $p,q \in \partial_{Euc}\Omega$, $p\ne q$ there exists a compact set $K\subset \Omega$ such that for any sequence $p_k, q_k \in \Omega$, with $p_k\to p$ and $q_k\to q$ there exists positive integer $k_0$ such that any quasihyperbolic geodesic $\gamma_k:[0,T]\to \Omega$ joining $p_k$ and $q_k$ ({\it i.e.,} with $\gamma_k(0)=p_k$ and $\gamma_k(T)=q_k$), we have
	$$\gamma_k([0,T])\cap K\ne\emptyset, \,\, \mbox{ for all } k\ge k_0.$$ 
	Geometrically, a bounded domain in $\mathbb{R}^n$ is a QH-visibility domain if the quasihyperbolic geodesics bend inside the domain when connecting points close to the Euclidean boundary $\partial_{Euc}\Omega$.
\end{rem}

\section{Results}
We are now going to introduce our results. 
\subsection*{Gromov hyperbolicity and visibility}
\addtocontents{toc}{\protect\setcounter{tocdepth}{1}}
We first see the behavior of Gromov product in QH-visibility domains.
\begin{thm}\label{gromov product visibility}
	Let $\Omega\subset \mathbb{R}^n$ be a bounded domain. Then the following are equivalent:
	\begin{itemize}
		\item[(1)] $\Omega$ is a QH-visibility domain.
		\item[(2)] For every pair of points $p,q\in \partial_{Euc}\Omega$, $p\ne q$, and sequences $z_k\to p$, $w_k\to q$, we have
		$$\limsup_{k\to \infty}(z_k|w_k)_o< \infty,$$
		for a fixed point (hence any) $o\in \Omega$.
	\end{itemize}
\end{thm} 
The proof of Theorem \ref{gromov product visibility} is similar to that of \cite[Proposition 2.4 and Proposition 2.5]{Bracci-2022}. Note that every Gromov hyperbolic space has visibility property in the sense of Eberlein and O'Neill when considering the Gromov boundary in the Gromov topology (see \cite[Part III, H, Lemma 3.2]{Bridson-book}). Therefore, it is clear that if $(\Omega,k_{\Omega})$ is Gromov hyperbolic and Gromov boundary is homeomorphic to the Euclidean boundary of $\Omega$, then $\Omega$ is a QH-visibility domain. In general, we prove the following result.

\begin{thm}\label{Main-thm-GB-EB}
	Let $\Omega\subset \mathbb{R}^n$ be a bounded domain. Assume that $(\Omega, k_D)$ is Gromov hyperbolic. Then the following are equiavlent
	\begin{itemize}
		\item[(1)] $\Omega$ is a QH-visibility domain 
		\item[(2)] the identity map $id:(\Omega,k_\Omega)\to(\Omega,d_{Euc})$ extends as a continuous surjective map $\Phi:\overline{\Omega}^G\to \overline{\Omega}^{Euc}$.
	\end{itemize}
\end{thm}

We also investigate conditions under which such an extension is a homeomorphism.

\begin{thm}\label{no loop}
	The extension $\Phi$ in Theorem \ref{Main-thm-GB-EB} is homeomorphism if, and only if, $\Omega$ has no geodesic loops in $\overline{\Omega}^{Euc}$.
\end{thm}

\begin{rem}\label{uniform domain visibility}
\begin{itemize}
	\item[1.] Theorem \ref{Main-thm-GB-EB} and Theorem \ref{no loop} together provide the complete solution to Problem \ref{GB-EB}. 
	\item[2.] In view of Theorem A and Theorem \ref{no loop}, we have that every bounded uniform domain is a QH-visibility domain with no geodesic loop in $\overline{\Omega}^{Euc}$. 
\end{itemize}
\end{rem}

Theorem \ref{no loop} and Theorem B together gives us the following important corollary.

\begin{cor}\label{Quasiconvex-no loop}
	Let $\Omega$ be a bounded domain in $\mathbb{R}^n$. Suppose $(\Omega,k_{\Omega})$ is Gromov hyperbolic. If $(\Omega, d_{Euc})$ is a  quasi-convex space, then $\Omega$ is a QH-visibility domain which has no geodesic loop in $\overline{\Omega}^{Euc}$.
\end{cor}

Theorem \ref{Main-thm-GB-EB} and Theorem \ref{no loop} are reminiscent of the result of Bracci {\it et. al.} \cite[Theorem 3.3]{Bracci-2022} in the context of visibility domains with respect to the Kobayashi distance, and we follow the similar technique for the proof in the context of quasihyperbolic metric. This technique can also be found in the work of Bonk, Heinonnen and Koskela \cite{Koskela-2001} and V\"ais\"al\"a \cite{Vaisala-GH-1}. Further, we define the {\it quasihyperbolically well behaved domain} to characterize no geodesic loop property.

\begin{defn}
A bounded domain $\Omega\subset\mathbb{R}^n$ is said to be quasihyperbolically well behaved if, $\{x_k\}$, $\{y_k\}$ are sequences in $\Omega$ with $$\lim_{k\to \infty}x_k=\lim_{k\to \infty}y_k=p\in\partial_{Euc}\Omega,$$
and $\gamma_k:[a_k,b_k]\to \Omega$ be a sequence of qh-geodesic joining $x_k$ and $y_k$, then we have
$$\lim_{k\to \infty}k_{\Omega}(x_0,\gamma_k)=\infty.$$
for some (hence all) $x_0\in \Omega$.
\end{defn}

We show that a quasihyperbolically well behaved domain does not have any geodesic loop. 
\begin{prop}\label{Proposition-1}
If $\Omega\subset\mathbb{R}^n$ is quasihyperbolically well behaved then $\Omega$ has no geodesic loop in $\overline{\Omega}^{Euc}$.
\end{prop}

Next we show that the converse of Proposition \ref{9.1}	holds under the visibility assumption.
\begin{prop}\label{Proposition-2}
	Let $\Omega\subset\mathbb{R}^n$ be a QH-visibility domain. Then $\Omega\subset\mathbb{R}^n$ is quasihyperbolically well behaved if, and only if, $\Omega$ has no geodesic loop in $\overline{\Omega}^{Euc}$.
\end{prop}

\subsection*{QH-Visibility domains} 
\addtocontents{toc}{\protect\setcounter{tocdepth}{1}}
Now, it is natural to ask which bounded domains in $\mathbb{R}^n$ are QH-visibility domains? Also we want to have a reasonably good collection of domains that are QH-visibility domains but not Gromov hyperbolic so that one can distinguish between these two notions of negative curvature. 
\par We first focus on the classical uniform domains and John domains. It is well-known in the Poincar\'e disk model of the hyperbolic space, hyperbolic geodesic between two points $x,y\in \mathbb{D}$, denote it by $\gamma_{hyp}[x,y]$, have the following properties:
\begin{itemize}
	\item[(i)] $l(\gamma_{hyp}[x,y])\le C|x-y|$, and
	\item[(ii)] $\min\{l(\gamma_{hyp}[x,z]),l(\gamma_{hyp}[z,y])\}\le C\,d(z,\partial \mathbb{D})$,
\end{itemize}
for all $z\in \gamma$, where $C=\pi/2$. The second condition is known as the cone arc condition by which we define the John domains (see Definition \ref{John}), which was first introduced by John \cite{John-1961} in the context of elasticity theory. If the curve satisfies both conditions mentioned above it is called double cone arc by which we define the uniform domains introduced by Martio and Sarvas \cite{Martio-1978} (see Definition \ref{uniform}). It is well-known that in uniform domains quasihyperbolic geodesics are cone arcs. In the John setting quasihyperbolic geodesics need not be double cone arcs. We observe that cone-arc condition on the quasihyperbolic geodesics in John domains gives us the visibility propert. Precisely, we prove the following result.

\begin{thm}\label{cone arc-visibility}
	Let $\Omega\subset \mathbb{R}^n$ be a bounded $c$-John domain. If every quasihyperbolic geodesic in $\Omega$ is a $b=b(c)$-cone arc then $\Omega$ is a visibility domain.
\end{thm}

\begin{rem}\label{inner uniform-visibility}
	As we know that every quasihyperbolic geodesic in Gromov hyperbolic John domain is cone arc \cite[Corollary 1.2]{Rasila-2022}, by Theorem \ref{cone arc-visibility}, we have that every Gromov hyperbolic bounded John domain is a QH-visibility domain. Also, in the Euclidean setting it is known that a John domain is Gromov hyperbolic if, and only if, it is inner uniform. In particular, every bounded inner-uniform domain is QH-visibility domain.
\end{rem}

We want to remove the condition of cone arc on quasihyperbolic geodesics from Theorem \ref{cone arc-visibility}. A careful observation shows that the quasihyperbolic metric in John domain satisfies a certain logarithmioc growth condition similar to Goldilocks domains. Thus, this motivates us to check the visibility of domains which have certain growth condition with respect to quasihyperbolic metric. Domain satisfing some growth conditions have great importance in geometric function theory. It is well-known that every $K$-quasiconformal map is locally H\"older continuous. To conclude the global H\"older continuity we need some extra geometric assumptions on the domians. Becker and Pommerenke \cite{BP-1988} have proved that If 
	$f:\mathbb{D} \to \Omega\subset \mathbb{C}$ is a conformal mapping, then $f$ is globally $\beta$ H\"older continuous, $0 <\beta\le1$, if and only if, the hyperbolic metric $h_{\Omega}$ in $\Omega$ satisfies a logarithmic growth condition 
	$$h_{\Omega}(f(0),z)\le \frac{1}{\beta}\,log \left(\frac{\delta_{\Omega}(f(0))}{\delta_{\Omega}(z)}\right)+C_0.$$
	Gehring and Martio \cite{Gehring-1985} have further extended this result for multiply connected domains and domains in higher dimensions. For a $K$-quasiconformal map $f:\Omega\to \Omega'$ Gehring and Martio have proved the global H\"older continuity of $f$ providing $\Omega$ is sufficiently nice (see \cite[p. 204]{Gehring-1985}) and $\Omega'$ satisfies the logarithmic growth condition with respect to quasihyperbolic metric.
	$$k_{\Omega'}(x_0,x)\le \frac{1}{\beta}\,log \left(\frac{\delta_{\Omega'}(x_0)}{\delta_{\Omega'}(x)}\right)+C_0.$$ 
	Thus, it is quite evident that studying visibility of such domains is of interest. Indeed, we prove the following result which we call as general visibility criteria. 
	
	\begin{thm}\label{visibility criteria}[General Visibility criteria]
		Let $\Omega\subset \mathbb{R}^n$ be a bounded domain. Fix a base point $x_0\in \Omega$ and $\phi:(0,\infty)\to(0,\infty)$ be an strictly increasing function with $\phi(t)\to \infty$ as $t\to \infty$ such that 
		\begin{equation}\label{int}
			\int_{0}^{\infty}\frac{dt}{\phi^{-1}(t)}<\infty
		\end{equation}
		If $\Omega$ satisfies the following growth condition 
		\begin{equation}\label{growth}
			k_{\Omega}(x_0,x)\le \phi\left(\frac{\delta_{\Omega}(x_0)}{\delta_{\Omega}(x)}\right),
		\end{equation}
		then $\Omega$ is a QH-visibility domain. 
	\end{thm}

\begin{rem}
	\begin{itemize}\label{John-visibility}
		\item[(i)] Theorem \ref{visibility criteria} gives us that domains satisfying quasihyperbolic boundary condition (see Definition \ref{QHBC-domains}) are QH-visibility domains. In particular, every bounded John domain is a QH-visibility domain (see Remark \ref{List}). This gives us a good collection of domains which are QH-visibility domains but not Gromov hyperbolic.
		\item[(ii)] Following Gotoh \cite{Gotoh-2000}, it is often convenient to write the  conditions in Theorem \ref{visibility criteria} in the following form 
		$$\delta_{\Omega}(x)\le \delta_{\Omega}(x_0)\psi(k_{\Omega}(x_0,x))$$
		where we assume $\psi:[0,\infty)\to (0,\infty)$ is non-increasing, and that
		$$\int_{0}^{\infty}\psi(t)dt<\infty.$$
	\end{itemize}
\end{rem}
 
	\subsection*{H-visibility domain and QH-visibility domain} 
	Note that in the case of planar hyperbolic domains the Kobayashi metric coincides with the hyperbolic metrc. It is natural to try to approach the visibility of planar domains with respect to hyperbolic metric from the corresponding results for the quasihyperbolic metric as these two some time similar but often quite different metric actually have quite similar geometry. We start with the following definition of hyperbolic visibility domain (H-visibility domain)
		\begin{defn}\label{H-visibility}
		Let $\Omega$ be a bounded hyperbolic domain in $\mathbb{C}$. Then $\Omega$ is said to be $H$-visibility domain if for every pair of points $p,q\in\partial_{Euc}\Omega$, $p\ne q$ there exists neighborhoods $U,V$ of $p,q$ respectively such that $\overline{U}\cap\overline{V}=\emptyset$ and a compact set $K_{\{p,q\}}\subset \Omega$ such that for any hyperbolic geodesic $\gamma:[0,T]\to \Omega$ with $\gamma(0)\in U$ and $\gamma(T)\in V$, we have 
		$$\gamma([0,T])\cap K\ne \phi.$$
	\end{defn}
	\begin{rem}\label{GH equivalence}
In 2020, Buckley and Herron \cite[Theorem B]{Buckley-2020} proved a marvolous result about Gromov hyperbolicity which says that a planar hyperbolic domain $\Omega$ is Gromov hyperbolic with respect to hyperbolic metric if, and only if, it is Gromov hyperbolic with respect to quasihyperbolic metric. Therefore, for planar hyperbolic domains we can simply say that $\Omega$ is Gromov hyperbolic without mentioning the metric.
	\end{rem} 
Motivated by this fact it is natural to consider the following problem
	\begin{prob}\label{H vs QH}
		Whether H-visibility and QH-visibility are equivalent for hyperbolic domains? In other words, does there exist a hyperbolic domain such that it is a visibility domain with respect to one metric but not the other? If such domain exists, the next problem is to describe this class of domains.
	\end{prob}
	
In this line we prove the following result which says that for Gromov hyperbolic domains H-visibility and QH-visibility are equivalent.

	\begin{thm}\label{GH-H-QH}
		Let $\Omega$ be a bounded hyperbolic domain. If $\Omega$ is Gromov hyperbolic, then the following are equivalent
		\begin{itemize}\label{H-QH}
			\item[(1)] $\Omega$ is H-visibility domain
			\item[(2)] $\Omega$ is a QH- visibility domain.
		\end{itemize}
	\end{thm}
	
	Further, in this line we prove that if a bounded planar domain satisfies the growth condition \eqref{growth} with respect to quasihyperbolic metric then it is a H-visibility domain. Observe that if a domain satisfies a growth condition with respect to quasihyperbolic metric, then it satisfies the same growth condition with respect to hyperbolic metric.
	
	\begin{thm}\label{growth-H-visibility}
		Let $\Omega\subset \mathbb{C}$ be a bounded domain. Fix a base point $x_0\in \Omega$ and $\phi:(0,\infty)\to(0,\infty)$ be an strictly increasing function with $\phi(t)\to \infty$ as $t\to \infty$ such that 
		\begin{equation}\label{int}
			\int_{0}^{\infty}\frac{dt}{\phi^{-1}(t)}<\infty
		\end{equation}
		If quasihyperbolic metric $k_\Omega$ satisfies the following growth condition 
		\begin{equation}
			k_{\Omega}(x_0,x)\le \phi\left(\frac{\delta_{\Omega}(x_0)}{\delta_{\Omega}(x)}\right),
		\end{equation}
		then $\Omega$ is a H-visibility domain. 
	\end{thm}
	
\begin{rem}
	Theorem \ref{growth-H-visibility} together with Theorem \ref{visibility criteria} gives us a large class of planar domains which are both H-visibility domain and QH-visibility domain. We list some examples of planar domains here
	\begin{itemize}
		\item[(a)] Bounded uniform domains.
		\item[(b)] Bounded John domains.
		\item[(c)] QHBC domains.
	\end{itemize}
\end{rem}

\subsection*{QH-visibility of domains quasiconformally equivalent to the unit ball} Further, we study visibility of domains which are quasiconformally equivalent to the unit ball. Observe that such domains are Gromov hyperbolic with respect to quasihyperbolic metric. Thus, to study visibility of such domains we use  Theorem \ref{Main-thm-GB-EB} and Theorem \ref{no loop}. As an application of Theorem \ref{Main-thm-GB-EB} and Theorem \ref{no loop} we prove following results
	
	\begin{thm}\label{LC along boundry}
		Let $\Omega\in \mathbb{R}^n$, $n\ge 2$ be a bounded domain which is quasiconformally equivalent to unit ball $\mathbb{B}^n$. Then following are equivalent
		\begin{itemize}
			\item[(i)] $\Omega$ is a QH-visibility domain
			\item[(ii)] $\Omega$ is finitely connected along the boundary
		\end{itemize}
		Moreover, $\Omega$ is QH-visibility domain with no geodesic loop in $\overline{\Omega}^{Euc}$ if, and only if, $\Omega$ is locally connected along the boundary.  
	\end{thm}
	
	In dimension $n=2$, since a domain is quasiconformally equivalent to unit ball ({\it i.e.}, unit disk) if, and only if, it is simply connected, therefore, Theorem \ref{LC along boundry} reduces to the following corollary.
	
	%we have the following simple characterization of visibility in the case of simply connected domains.
	
	\begin{cor}\label{LC boundary}
		Let $\Omega\subsetneq \mathbb{C}$ be simply connected. Then the following are equivalent
		\begin{itemize}
			\item[(i)] $\Omega$ is a QH-visibility domain
			\item[(ii)] $\partial_{Euc}\Omega$ is locally connected.
		\end{itemize}
		Moreover, $\Omega$ is a QH-visibility domain with no geodesic loop in $\overline{\Omega}^{Euc}$ if, and only if, $\Omega$ is locally connected along the boundary, or simply, $\Omega$ is a Jordan domain.
	\end{cor}
	
%	\begin{rem}
	%	By a result of Wilder \cite[pp. 66]{Wilder}, that every Jordan domain is locally connected along the boundary and by \cite[Theorem 17.22]{Vaisala-Book} we have that a domain which is quasiconformally equivalent to a ball is Jordan domain if, and only if, it is locally connected along the boundary.
%	\end{rem}
	
	\subsection{Continuous extension of quasihyperbolic quasi-isometries and isometries} The main motivation to study continuous extension of quasi-isometry is the fact that every quasiconformal map is a rough quasi-isometry with respect to quasihyperbolic metric \cite[Theorem 3]{Gehring-1979}, and continuous extension of quasiconformal map is always of some interest. For more details on the continuous extension of quasiconformal mappings we refer to \cite{Vaisala-Book}. We first define a domain which we call as quasihyperbolically well behaved which is motivated by the no geodesic loop property (see Section \ref{Isometry and quasiisometry}). We prove the following two results related to isometries and $(1,\mu)$-quasihyperbolic quasi-isometry.

	\begin{thm}\label{Isometry}
		Suppose $\Omega\subset\mathbb{R}^n$ is quasihyperbolically well behaved domain and $\Omega'\subset\mathbb{R}^n$ is a visibility domain. Then every quasihyperbolic isometry $f:(\Omega,k_\Omega)\to (\Omega,k_{\Omega})$ extends to a continuous map $F:\overline{\Omega}^{Euc}\to \overline{\Omega'}^{Euc}$.
	\end{thm}
	
\begin{thm}\label{quasiisometry}
Suppose $\Omega\subset\mathbb{R}^n$ is quasihyperbolically well behaved domain and $\Omega'\subset\mathbb{R}^n$ is a visibility domain. Then every continuous $(1,\mu)$-quasihyperbolic quasi-isometric embedding $f:(\Omega,k_\Omega)\to (\Omega',k_{\Omega'})$ extends to a continuous map $F:\overline{\Omega}^{Euc}\to \overline{\Omega'}^{Euc}$.
\end{thm}

Further as an application of Theorem \ref{Main-thm-GB-EB} and Theorem \ref{no loop}, we prove that every quasiconformal map has a homeomorphic extension to the Euclidean boundary under the condition that domains are Gromov hyperbolic and visibility domains with no geodesic loops.

\begin{thm}\label{quasiconformal}
Let $\Omega$ and $D$ be bounded domains in $\mathbb{R}^n$. If $\Omega$ and $D$ are both Gromov hyperbolic and visibility domain with no geodesic loops, then any quasi-isometry $f:(\Omega,k_{\Omega})\to (D,k_D)$ which is homeomorphism extends to a homeomorphism $F:\overline{\Omega}^{Euc}\to \overline{D}^{Euc}$. In particular, every quasiconformal map $f:\Omega\to D$ has a homeomorphic extension to the Euclidean boundary.
\end{thm}

\subsection{QH-visibility for unbounded domains in $\mathbb{R}^n$} We want to study the visibility property of unbounded domains with respect to the quasihyperbolic metric so the visibility of unbounded domains with respect to hyperbolic and quasihyperbolic metrics can be compared. Bharali and Zimmer \cite{Bharali-2022} use the end compactification to define the visibility of unbounded domains with respect to the Kobayashi metric. We use the term BZ-visibility domain. In this line, Masanta \cite[Theorem 1.16]{Rumpa-2024} and Chandel {\it et al.} \cite[Theorem 1.6]{Sarkar-2024} independently proved an important result which says that to check the BZ-visibility of a domain in $\mathbb{C}^n$ it is enough to check that visibility property holds true for Euclidean boundary points, {\it i.e.} there is no need to verify the condition at the ends of $\overline{\Omega}^{Euc}$. Motivated by this we define the visibility of unbounded domains with respect to quasihyperbolic metric as follows

\begin{defn}\label{QH-visible}
		Let $\Omega\subset \mathbb{R}^n$ be a domain (not necessarily bounded). We say that a pair $\{p,q\}\subset\partial_{Euc}\Omega$ of distinct points is QH-visible if there exist 
		\begin{itemize}
			\item neighborhoods $U_p$ of $p$ and $U_q$ of $q$ such that $\overline{U}\cap\overline{V}=\emptyset$, 
			\item and a compact set $K_{\{p,q\}}\subset \Omega$ 
		\end{itemize}
	such that for any quasihyperbolic geodesic $\gamma:[0,T]\to \Omega$ with $\gamma(0)\in U_p$ and $\gamma(T)\in U_q$, we have 
		$$\gamma([0,T])\cap K\ne \emptyset.$$
		Fruther, we say $\partial_{Euc}\Omega$ is {\it QH-visibile} if every every pair of distinct points in $\partial_{Euc}\Omega$ is QH-visible. 
\end{defn}

\begin{defn}[Unbounded QH-visibility domain]\label{unbounded visibility}
Let $\Omega\subset \mathbb{R}^n$ be a domain (not necessarily bounded). We say $\Omega$ is a QH-visibility domain if $\partial_{Euc}\Omega$ is {\it QH-visibile}.
\end{defn}

\begin{thm}\label{unbounded QH-Gromov}
	Let $\Omega\subsetneq \mathbb{R}^n$, $n\ge 2$ be Gromov hyperbolic with respect to quasihyperbolic metric. Then following are equivalent
	\begin{itemize}
		\item[(i)] $\Omega$ is a QH-visibility domain
		\item[(ii)] $id_{\Omega}$ extends to a continuous surjective map $\widehat{id_{\Omega}}: \overline{\Omega}^G\to \overline{\Omega}^{\infty}$, where $\overline{\Omega}^{\infty}$ denote the one-point compactification of $\overline{\Omega}^{Euc}$.
	\end{itemize}
In particular, every unbounded unbounded uniform domain is a QH-visibility domain.
\end{thm}

Theorem \ref{unbounded QH-Gromov} gives us the following important corollary. 
\begin{cor}\label{Simply connected}
	Let $\Omega\subsetneq \mathbb{C}$ be simply connected. Then following are equivalent
	\begin{itemize}
		\item[(i)] $\Omega$ is a QH-visibility domain 
		\item[(ii)] $\partial^{\infty}\Omega$ is locally connected subset of $\widehat{\mathbb{C}}$.
	\end{itemize}
\end{cor}

\begin{rem}
	One can easily check that Theorem \ref{GH-H-QH} also holds true for unbounded planar hyperbolic domains, and hence Corollary \ref{Simply connected} gives us that {\it a planar simply connected domain is a H-visibility domain if, and only if, the extended boundary is locally connected in the extended complex plane} (see \cite[Theorem 1.12]{Sarkar-2024})
\end{rem}

	\section{Preliminaries}\label{Preliminaries}
	\subsection{Metric geometry} Let $(X,d)$ be a metric space. A curve is a continuous function $\gamma:[a,b]\rightarrow X$. Let $\mathcal{P}$ denote set of all partitions $a = t_{0}<t_{1}<t_{2}< \cdots<t_{n}=b$ of the interval $[a,b]$. The  length of the curve $\gamma$ in the metric space $(X,d)$ is 
	$$l_{d}(\gamma) = \sup_{\mathcal{P}} \sum_{k=0}^{n-1}d(\gamma(t_{k}), \gamma(t_{k+1})).$$
	A curve is said to be rectifiable if $l_{d}(\gamma) < \infty$.  A metric space $X$ is said to be rectifiably connected if every pair of points $x,y \in X$ can be joined by a rectifiable curve. For a rectifiable curve $\gamma$ we define arc length $s:[a,b]\rightarrow \left[ 0,l_d(\gamma) \right] $ by $s(t) = l_{d}(\gamma|_{[a,t]})$. The arc length function is of bounded variation. For any rectifiable curve $\gamma:[a,b]\to X$, there is a unique map $\gamma_s:[0,l_d(\gamma)]\to X$ such that $\gamma=\gamma_s\circ s$, and such a curve $\gamma_s$ is called the arclength parametrization of $\gamma$.\\
	
	For $x,y\in X$, the inner length metric $\lambda_X(x,y)$ is defined by
	$$\lambda_X(x,y)=\inf\{l_d(\gamma): \gamma \mbox{ is a rectifiable arc joining } x \mbox{ and } y\}.$$
	
	Let $\rho :X\rightarrow[0,\infty]$ be a Borel function. The $\rho$-length of a rectifiable curve $\gamma$ is $$l_{\rho}(\gamma)=\int_{\gamma}\rho\,ds=\int_{a}^{b}\rho(\gamma(t))\, ds(t)=\int_{0}^{l_d(\gamma)}\rho\circ \gamma_s(t)\,dt.$$
	
	If $X$ is rectifiably connected, then the $\rho$-distance between two points $x,y\in X$ is 
	$$d_{\rho}(x,y) = \inf\{l_{\rho}(\gamma): \gamma \mbox{ is a rectifiable curve joining }    x,y \mbox{ in } X\},$$
	where infimum is taken over all rectifiable curves joining $x$ and $y$ in $X$. In general, $d_\rho$ need not be a metric since  $d_\rho(x,y)$ could be zero or infinite. In order to ensure that $(X,d_{\rho})$ is a metric space, we must require that $0<d_{\rho}(x,y)<\infty$. To differentiate this, we call such $\rho$ a metric-density or simply, density on $X$. However, if $\rho$ is positive and continuous then $d_{\rho}$ is a metric and, we say $(X,d_\rho)$ is a conformal deformation of $(X,d)$ by the conformal factor $\rho$.
	
\begin{defn}\label{quasi-convex}
A metric space $(X,d)$ is said to be quasi-convex if there exists a constant $A\ge 1$ such that every pair of points $x,y\in \Omega$ can be joined by a curve $\gamma$ such that
$$l_d(\gamma)\le Ad(x,y).$$ 
\end{defn}
	
	\subsection*{Maps, geodesics and quasigeodesics}
	\begin{defn}[\textbf{bi-Lipschitz}]
		Let $(X,d)$ and $(Y,d')$ be metric spaces. We say that a map $f:X\to Y$ is bi-Lipschitz if there is a constant $\lambda\ge 1$ such that for all $x,y\in X$, we have
		\begin{equation}\label{bi-Lipschitz}
			\lambda^{-1}d(x,y)\le d'(f(x),f(y))\le \lambda d(x,y),
		\end{equation}
		and when this holds we say that $f$ is $\lambda$-bi-Lipschitz.
	\end{defn}
	\begin{defn}[\textbf{Quasiisometry}]
		Let $\lambda\ge 1$ and $\mu\ge 0$ and $(X,d)$ and $(Y,d')$ be metric spaces. We say that a map (not necessarily continuous) $f:X\to Y$ is a $(\lambda,\mu)$-quasiisometry if for all $x,y\in X$, we have
		\begin{equation}\label{quasi-isometry}
			\lambda^{-1}d(x,y)-\mu\le d'(f(x),f(y))\le \lambda d(x,y)+\mu.
		\end{equation}	
		A $(1,0)$ quasiisometry is simply an isometry (onto its range). 
	\end{defn}
	\begin{rem}
		\begin{itemize}
			\item[(1)] We say that two metric spaces $X,Y$ are {\it ($\lambda,\mu$)-quasiisometrically equivalent} if there is a ($\lambda,\mu$)-quasiisometry $f:X\to Y$ with the property that for every $y\in Y$ there is an $x\in X$ with $d'(y,f(X))\le \mu$. An alternative way to describe this is to say that there are quasiisometries in both directions that are rough inverses of each other (see \cite[Exercises 8.16 (2)]{Bridson-book}).
			\item[(2)] It is important to note that some authors use the adjective rough quasiisometry to satisfy our definition of quasiisometry.
		\end{itemize}
	\end{rem}

	\begin{defn}
	A curve $\gamma:[a,b]\to X$ is said to be geodesic if it is an isometry {\it i.e.,} for all $t,t'\in [a,b]$, $$d(\gamma(t), \gamma(t'))=|t-t'|.$$
	An isometry $\gamma:[0,\infty)\to X$ is called a geodesic ray, and an isometry $\gamma:(-\infty,\infty)\to X$ is called a geodesic line. 
\end{defn}

\begin{defn}[\textbf{Quasi-geodesic}]
	Let $(X,d)$ be a metric space. A curve $\gamma:[a,b]\to X$ is said to be $\lambda$-{\it quasigeodesic} (($\lambda,\mu$)-quasi-geodesic), if $\gamma$ is $\lambda$-bi-Lipschitz  (($\lambda,\mu$)-quasiisometry).
	There is another corresponding description of quasi-geodesics: a curve $\gamma:[a,b]\to X$ is said to be $L$-{\it chordarc curve} if
	$\mbox{for all }s,t \in [a,b],\,\,\,\,\, l_d(\gamma|_{[s,t]})\le L\,d(\gamma(s),\gamma(t))$.
\end{defn}
\begin{rem}\label{quasigeodesic and chord arc}
	We note that a $\lambda$-quasi-geodesic is a $\lambda^2$-chordarc curve, and the arc length parametrization of a $L$-chordarc curve is a $L$-quasi-geodesic.
\end{rem}
\begin{defn}\label{GPI}
	A metric space $(X,d)$ is said to be
	\begin{itemize}
		\item[(1)] {\it a geodesic space} if every pair of points $x,y \in X$ can be joined by a geodesic,
		\item[(2)] {\it proper} if every closed ball in $(X,d)$ is compact in $(X,d)$,
			\item[(3)] intrinsic or path or length space if for all $x,y\in X$, we have
		$$d(x,y)=\lambda_X(x,y).$$
	\end{itemize}
\end{defn}
It is clear that if a metric space is proper, then it is locally compact. In general, we have the following well-known result attributed to H. Hopf and W. Rinow  (see \cite[I, Proposition 3.7]{Bridson-book} or \cite[Theorem 2.1.15 and Theorem 2.4.6]{Papadopoulos-book}), which gives the connection between spaces mentioned in Definition \ref{GPI}
\begin{thm}[Hopf-Rinow theorem]
	Let $(X,d)$ be an intrinsic space. Then following holds
	\begin{itemize}
		\item[(1)] If $(X,d)$ is complete and locally compact, then it is proper.
		\item[(2)] If $(X,d)$ is proper, then it is a geodesic space.
		\item[(3)] If $(X,d)$ is locally compact, then it is proper if and only if it is complete.
	\end{itemize}
\end{thm}
\begin{rem}
	\begin{itemize}
		\item[(1)] For further discussion on Hopf-Rinow theorem, we refer to \cite[Notes on Chapter 2]{Papadopoulos-book} and references therein.
		\item[(2)] It is worth to mention that \cite[Theorem 2.8]{Martin-1985} has proved that if $\Omega$ is a proper subdomain of $\mathbb{R}^n$ and $\rho$ is a continuous density on $\Omega$, and it defines a complete metric space $(\Omega,d_{\rho})$, then $(\Omega,d_{\rho})$ is a geodesic metric space.
	\end{itemize}
\end{rem}

	\subsection{Hyperbolic and quasihyperbolic metric} 
	
	\begin{defn}[\textbf{The Poincar\'e metric}]
		The Poincar\'e metric $\lambda_{\mathbb{D}}(z)|dz|$ in the unit disk $\mathbb{D}$ is defined by the length element	
	$$\frac{|dz|}{1-|z|^2},\quad \lambda_{\mathbb{D}}(z)=\frac{1}{1-|z|^2}.$$
	Thus, the Poincar\'e/hyperbolic distance $h_{\mathbb{D}}$ in $\mathbb{D}$ is given by 
	\begin{equation}\label{Poincare distance}
		h_{\mathbb{D}}(z_1,z_2) = \inf_{\gamma} l_h(\gamma), \mbox{ where } l_h(\gamma)=\int_{\gamma}\frac{|dz|}{1-|z|^{2}}
	\end{equation}
	is the hyperbolic length of a rectifiable curve $\gamma$ in $\mathbb{D}$, and the infimum is taken over all rectifiable curves $\gamma$ joining $z_1$ and $z_2$.
	\end{defn} 
	
A domain $\Omega \subsetneq\mathbb{C}$ is said to be hyperbolic if $\mathbb{C}\setminus\Omega$ contains atleast two points. It is known that the unit disk covers every hyperbolic domain, {\it i.e.,} there is a holomorphic universal covering map $h:\mathbb{D}\to \Omega$. We define the hyperbolic metric on any hyperbolic domain $\Omega$ via the pull-back by $h$. For details on pull-back we refer to \cite[Definition 9.2]{Beardon-Minda-notes}.
\begin{thm}\cite[Theorem 10.3]{Beardon-Minda-notes}
	Each hyperbolic domain carries a unique maximal metric $\lambda_{\Omega}(w)|dw|$ such that the pull back of $\lambda_{\Omega}(w)|dw|$ by $h$ is $\lambda_{\mathbb{D}}(z)|dz|$. The metric $\lambda_{\Omega}(w)|dw|$ has constant curvature $-4$. 
\end{thm}
Note that some authors prefer to use $2\lambda_{\Omega}$ which has Gaussian curvature $-1$. We call $\lambda_{\Omega}(z)$ the hyperbolic density which induces the metric $h_{\Omega}$ on $\Omega$ called the hyperbolic metric or hyperbolic distance defined by
\begin{equation}\label{hyperbolic distance}
	h_{\Omega}(w_1,w_2)=\inf_{\gamma}\int_{\gamma}\lambda_{\Omega}(w)|dw|,
\end{equation}
where the infimum is taken over all the rectifiable curve $\gamma$ joining $w_1,w_2\in \Omega$. The topology induced by $h_{\Omega}$ on $\Omega$ coincides with the Euclidean topology. Moreover, $(\Omega, h_{\Omega})$ is a complete metric space. For more details on hyperbolic metric and its connecten with geometric function theory we refer to \cite{Beardon-Minda-notes}.
	\subsection*{Hyperbolic Geodesic.} 	A curve $\gamma$ in $\Omega$ joining $x$ and $y\in \Omega$ is said to be hyperbolic geodesic if
	\begin{equation*}
		l_h(\gamma) = h_{\Omega}(x,y)
	\end{equation*}
Hyperbolic geodesics in the unit disk are arc of circle orthogonal to the unit circle or the straight line segment joining $0$ and $z$, where $z\in \mathbb{D}$ is a non-zero point. Hyperbolic geodesics in the uper half plane are arcs of semi circles orthoganl to the real axis and verticle line segments.

\subsection*{The quasihyperbolic metric} Gehring and Palka studied quasiconformally homogeneous domains with a hope that this may eventually lead to a new characterization for domains quasiconformally equivalent to the unit ball $B^n$.  It is easy to observe that such domains are homogeneous with respect to quasiconformal family. Gehring and Palka Gehring and Plaka \cite{Gehring-1976} have introduced quasihyperbolic metric as a tool to study quasiconformal homogeneity, and proved that the maximal dilatation of such quasiconformal mappings can be estimated in terms of quasihyperbolic metric and using these estimates they obtained useful results related to the homogeneity of domains with respect to quasiconformal family. Since then this metric has found numbers of applications. We now give a precise definition of the quasihyperbolic metric.
\begin{defn}
	Let $\Omega\subsetneq\mathbb{R}^n$ be a domain. We define the {\it quasihyperbolic metric} $k_{\Omega}$ in $\Omega$ by the length element
	$$\frac{|dx|}{\delta_{\Omega}(x)},\,\,\,\,\, \delta_{\Omega}(x)=d_{Euc}(x, \partial_{Euc} \Omega),$$
	where $|dx|$ dentoes the Euclidean length element. Thus, 	$k_{\Omega}(x,y)=\inf_{\gamma}l_{k}(\gamma)$, where the infimum is taken over all rectifiable curves $\gamma$ in $\Omega$ joining $x$ and $y$, and
	$$l_{k}(\gamma)=\int_{\gamma}\frac{|dx|}{\delta_{\Omega}(x)}$$ 
	is called the {\it quasihyperbolic length} of a rectifiable curve $\gamma$ in $\Omega$.
\end{defn}

It is clear that in the upper half-space $\delta_{\Omega}(x)^{-1}|dx|=x_n^{-1}|dx|$, which is the Riemannian metric associated to the hyperbolic metric. Thus, in the upper half-space the hyperbolic and quasihyperbolic metric coincides. This motivates the term {\it quasihyperbolic}. We have the following comparison results between hyperbolic and quasihyperbolic metric.
	\begin{itemize}
		\item[(i)] $k_{\Omega}=h_{\Omega}$ when $\Omega$ is a half space, and  $h_{\Omega}\le 2k_{\Omega}$ for every hyperbolic domain $\Omega$.
		\item[(ii)] $k_{\Omega}\le h_{\Omega}\le 2k_{\Omega}$, when $\Omega$ is a ball in $\mathbb{R}^n$.
		\item[(iii)] $k_{\Omega}/2\le h_{\Omega}\le 2k_{\Omega}$, when $\Omega$ planar simply connected domain.
		\item[(iv)] In general, these two metrics are comparable if, and only if, the boundary is uniformly perfect (see \cite{BP-1978,P-1979})
	\end{itemize}
For further geometric properties of quasihyperbolic metric, we refer to \cite{Gehring-1979,Matti-book}. 	
\subsection*{Quasi-hyperbolic geodesic}
	A rectifiable curve $\gamma\subset\Omega$ is said to be a quasihyperbolic geodesic if $k_{\Omega}(x,y)=l_{k}(\gamma|_{[x,y]})$, for each pair of points $x,y\in \gamma$. Obviously each subarc of quasihyperbolic geodesic is again a geodesic. Gehring and Osgood \cite[Lemma 1]{Gehring-1979} have proved the existence of quasihyperbolic geodesics between any two points of the domain $\Omega$. For important results on quasihyperbolic geodesics in domain in $\mathbb{R}^n$, we refer to Martin \cite{Martin-1985}. Moreover, Martin \cite[Corollary 4.8]{Martin-1985} has proved that quasihyperbolic geodesics in $\mathbb{R}^n$ are $C^{1,1}$ smooth {\it i.e.,} the arc length parametrization has Lipschitz continuous derivatives.
		
	\subsection*{Estimates on quasihyperbolic metric} We recall some basic estimates for the quasihyperbolic metric which have been first introduced by Gehring and Palka \cite[2.1]{Gehring-1976} in $\mathbb{R}^n$.  Let $\Omega\subsetneq \mathbb{R}^n$ be a domain and $x,y\in \Omega$ and let $\gamma$ be a rectifiable curve in $\Omega$ joining $x$ and $y$. Then we have the following:
	\begin{equation}\label{qh-eq-1}
		k_{\Omega}(x,y)\ge log\left(1+\frac{\lambda_{\Omega}(x,y)}{\min\{\delta(x),\delta(y)\}}\right)\ge log\left(1+\frac{|x-y|}{\min\{\delta(x),\delta(y)\}}\right)\ge \left|log \frac{\delta(y)}{\delta(x)}\right|
	\end{equation}
	and 
	\begin{equation}\label{qh-eq-2}
		l_k(\gamma)\ge log\left(1+\frac{l(\gamma)}{\min\{\delta(x),\delta(y)\}}\right).
	\end{equation}

	The inequality \ref{qh-eq-1} implies that
	\begin{equation}\label{qh-eq-3}
		\lim_{y\to \partial_{Euc}\Omega}k_{\Omega}(x,y)=\infty
	\end{equation}
	for $x\in \Omega$. Using \eqref{qh-eq-3}, we have that $(\Omega, k_{\Omega})$ is a complete metric space. Moreover, it induces the usual topology \cite[Corollary 2.2]{Gehring-1976}. In view of \cite[Proposition 2.8]{Koskela-2001}\label{Koskela}, it is well-known that $(\Omega,k_\Omega)$ is a complete, proper and geodesic metric space. The following result is a simple fact which we will frequently use in our proofs.
	
\subsection*{Result 1} \cite[Corollary 2.3]{Gehring-1976} Let $\Omega$ be a proper subdomain in $\mathbb{R}^n$. Then for each set $F\subset\Omega$, $k_{\Omega}(F)=\sup_{x,y\in F}k_{\Omega}(x,y)<\infty$ if, and only if, $\overline{F}^{Euc}\subset \Omega$.

	The following observation immediately follows from \eqref{qh-eq-1}.
	\begin{observation}\label{observation}
		Let $\Omega$ be a bounded domain. If sequences $\{z_k\}$, $\{w_k\} \subset \Omega$ converge to two different points on $\partial_{Euc}\Omega$, then 
		$$\lim_{k\to \infty}k_{\Omega}(z_k,w_k)=\infty.$$
	\end{observation}

	\par Throughout this paper, we denote, in short, quasihyperbolic geodesic as qh-geodesic whenever required.
	
	\subsection{Gromov hyperbolic space}\label{GHS}  Let $(X,d)$ be a geodesic metric space, a geodesic triangle is the union of geodesics $\gamma_i:[a_i,b_i]\to X$, $i=1,2,3$ such that $a_i< b_i$ for every $i=1,2,3$ and $\gamma_1(b_1)=\gamma_2(a_1), \gamma_2(b_2)=\gamma_3(a_3), \gamma_3(b_2)=\gamma_1(a_1)$. The geodesics $\gamma_1, \gamma_2$ and $\gamma_3$ are called sides of the geodesic triangle.
	
	\begin{defn}
		Let $\delta\ge 0$. A geodesic metric space $(X,d)$ is said to be $\delta$-hyperbolic if every geodesic triangle is $\delta$-thin. That is, each side of the geodesic triangle is contained in the $\delta$-neighborhood of the union of the other two sides.
	\end{defn}
	
	\begin{defn}
		A geodesic metric space $(X,d)$ is said to be Gromov hyperbolic if it is $\delta$-hyperbolic for some $\delta\ge 0$.
	\end{defn}
	
	\subsection*{Examples} Clearly, any bounded geodesic metric space is Gromov hyperbolic. Clearly $\mathbb{R}$ is Gromov hyperbolic with $\delta=0$. For $n\ge 2$, $\mathbb{R}^n$ is not Gromov hyperbolic. The hyperbolic space $\mathbb{H}^n$ is Gromov hyperbolic with $\delta=\log3$. Moreover, every CAT($k$) space, $k<0$ is Gromov hyperbolic for $\delta=(\log\,3)/\sqrt{-k}$. In particular, every complete simply connected Riemannian manifold with sectional curvature bounded from above by $- k$, with $k > 0$, is Gromov hyperbolic. One can construct more examples by considering the Gromov hyperbolicity of Riemann surfaces, for this we refer to \cite{Rodriguez-2006}. In particular, some classical surfaces such as punctured unit disk and the annuli with Poincar\'e metric are Gromov hyperbolic (for the bound on the hyperbolicity constant see \cite[Theorem 5.1]{Rodriguez-2006}.
	
		\subsection*{Geodesics and quasigeodesics in Gromov hyperbolic spaces} 
One of the important propetry of Gromov hyperbolicity is that it is invariant under quasi-isometries. The proof of this result is based on the stability of geodesics in $\delta$-hyperbolic spaces. Geodesics in Gromov hyperbolic spaces exhibit some stability and are generally well-behaved.  which we will discuss next. In particular, quasi-geodesics and geodesics sharing the same endpoints stay close to each other.
	
	\begin{defn}
		Let $X$ and $Y$ be two non-empty subsets of a metric space $(M,d)$. We define their Hausdorff distance $d^{\mathcal{H}}(X,Y)$ by 
		$$d^{\mathcal{H}}(X,Y)=\max\left\{\sup_{x\in X}d(x,Y),\, \sup_{y\in Y}d(X,y)\right\},$$
		where $d(a,B)=\inf_{b\in B}d(a,b)$ denotes the distance from a point $a\in M$ to a subset $B\subset M$.
	\end{defn}

	\begin{thm}\label{morse}(\textbf{Shadowing lemma/ Geodesic stability})
		Suppose $(X,d)$ be a Gromov hyperbolic space with $\delta\geq0$. For any $\lambda \geq 1, \mu \geq 0$ there exists a constant $M=M(\delta, \lambda, \mu)$ such that: If $\gamma:[a_1,b_1]\to X$ is a $(\lambda, \mu)$-quasi-geodesic and $\sigma:[a_2,b_2]\to X$ is a geodesic with $\gamma(a_1)=\sigma(a_2)$, $\gamma(a_2)=\sigma(a_2)$, then $$d^{\mathcal{H}}(\gamma([a_1,b_1]), \sigma([a_2,b_2]))\le M,$$ 
	\end{thm}

This says that Gromov hyperbolic spaces are geodesically stable in the sense that each quasigeodesic segment is contained in a neighborhood of a geodesic segment. The converse is also true and is due to Bonk \cite{Bonk-1995}.	Theorem \ref{morse} has the following important consequence which says that one can define Gromov hyperbolicity in terms of  $(\lambda, \epsilon)$-quasi-geodesic triangles.
	
	\begin{cor}
		A geodesic metric space $X$ is hyperbolic if and only if, for every $\lambda \geq 1$ and every $\mu \geq 0$, there exists a constant $M$ such that every $(\lambda, \epsilon)$-quasi-geodesic triangle in $X$ is $M$-thin. (If $X$ is $\delta$-hyperbolic, then $M$ depends only on $\delta, \lambda$ and $\epsilon$.)
	\end{cor}	
	
	Further, Theorem \ref{morse} gives the invariance of Gromov hyperbolicity under quasi-isometry.		
	
	\begin{thm}\label{thm2}
		Let $X$ and $Y$ be geodesic metric spaces and let $f : X \rightarrow Y$ be a quasi-isometric embedding. If $Y$ is Gromov hyperbolic then $X$ is Gromov hyperbolic. ({\it i.e.,} If $Y$ is $\delta$-hyperbolic and $f : X \rightarrow Y$ is a $(\lambda, \mu)$-quasi-isometric embedding, then $X'$ is $\delta'$-hyperbolic, where $\delta'$ depends only on $\delta, \lambda$ and $\mu$). In addition, if $f(X)=Y$ then $X$ is Gromov hyperbolic if, and only if, $Y$ is Gromov hyperbolic
	\end{thm}
	
	\begin{rem}
		Theorem \ref{thm2} can be used to construct non-examples of Gromov hyperbolic spaces. For instance, consider the two-dimensional jungle-gym (a $\mathbb{Z}^2$ covering of a Riemann surface of genus $\ge 2$) equipped with a metric of constant negative curvature. This is quasi-isometric to $\mathbb{R}^2$ and hence not Gromov hyperbolic. Infact, geodesic triangles are nearly Euclidean.
	\end{rem}
	
	\subsection{Gromov compactification}
	For the rest of this subsection we assume that $(X,d)$ is a proper geodesic Gromov hyperbolic metric space. Let $p\in X$ and $\mathcal{G}_{p}$ be the space of all geodesic rays $\gamma:[0,\infty)\to X$ with $\gamma(0)=p$, endowed with the topology of uniform convergence on compact subsets of $[0,\infty)$. Define an equivalence relation $\sim$ on $\mathcal{G}_{p}$ as follows: Let $\gamma,\sigma\in \mathcal{G}_{p}$ be two geodesic rays. Then
	$$\gamma\sim \sigma \iff \sup_{t\ge 0} d(\gamma(t),\sigma(t))< \infty.$$ 
	
	\begin{defn}
		\begin{itemize}
			\item[(1)] The {\it Gromov boundary} of $X$ is denoted by $\partial_{G} X$ and is defined as the quotient space $\mathcal{G}_{p}/\sim$ endowed with the quotient topology.
			\item[(2)]  The {\it Gromov closure} of $X$ is denoted by $\overline{X}^{G}$ and is defined as $\overline{X}^{G}=X\cap \partial_{G}X$.
		\end{itemize}
	\end{defn}

\begin{rem}
The definition of the Gromov boundary does not depend on the base point. Indeed for given $p,q\in X$, there is a natural map $J: \mathcal{G}_p/\sim\, \to \mathcal{G}_q/\sim$. Let $[\gamma]\in \mathcal{G}_p$, where $\gamma:[0,\infty)\to X$ is a geodesic ray with $\gamma(0)=p$. For each $n\in\mathbb{N}$, consider a sequence of geodesics $\gamma_n:[0,R_n]\to X$ with $\gamma_n(0)=q$ and $\gamma_n(R_n)=\gamma(n)$. Then by Ar\'zela-Ascoli's theorem, up to a subsequence, we can assume that $\gamma_n$ converges uniformly on every compact subset to a geodesic ray $\sigma:[0,\infty)\to X$ with $\gamma(0)=q$. Now define, $J([\gamma])=[\sigma]$. Since $(X,d)$ is Gromov hyperbolic, $J$ is well defined and onto. Infact $J$ is a homeomorphism (see \cite[Chapter III.H.3, Lemma 3.3, Lemma 3.6 and Proposition 3.7]{Bridson-book}).
\end{rem}

	\subsection*{Gromov topology} The set $\overline{X}^{G}$  admits a natural topology and with this topology $\overline{X}^{G}$ is a compactification of $X$, known as Gromov compactification of $X$ (see for instance \cite[Chapter III.H.3, Proposition 3.7]{Bridson-book}).\\
	
	 To understand this topology we need some additional notation: A generalized ray is either a geodesic ray or a geodesic. For a geodesic ray $\gamma \in \mathcal{G}_p$ define $\mbox{End}(\gamma)=\{\sigma:[0,\infty)\to X: \sigma \mbox{ is a geodesic ray with } \gamma(0)=p \mbox{ and } \sigma \sim \gamma\}=[\gamma]$, and for a geodesic $\gamma:[0,R]\to X$ with $\gamma(0)=p$, where $R>0$, define $\mbox{ End}(\gamma)=\gamma(R)$. Then a sequence $\{x_n\}_{n\in \mathbb{N}}$ in $\overline{X}^G$ converges to $x\in \overline{X}^G$ (denoted by $x_n \stackrel{Gromov}\longrightarrow x$), if, and only if, for every generalized ray $\gamma_n$ with $\gamma_n(0)=p$ and $\mbox{ End}(\gamma_n)=x_n$ its every subsequence has a subsequence which converges uniformly on every compact subset to a generalized ray $\gamma$ with $\gamma(0)=x_0$ and $\mbox{ End}(\gamma)=x$.\\

	\subsection{Gromov product}
	There is an alternate description of Gromov hyperbolic spaces via Gromov product which is useful in the following sense
	\begin{itemize}
		\item[(i)] This notion allows us to define Gromov hyperbolicity for arbitrary metric space (not necessarily geodesic or proper (see \cite{Vaisala-GH})).
		\item[(ii)]This notion allows us to define Gromov boundary for arbitrary metric space (not necessairly Gromov hyperbolic)
		\item[(iii)] This notion allows us to define a canonical gauge of metrics on Gromov boundary, which are known as visual metrics.
	\end{itemize}
	
\begin{defn}
Fix a base point $o\in X$ and let $x,y\in X$. The Gromov product of $x$ and $y$ is defined as
$$(x|y)_o=\frac{1}{2}(d(x,o)+d(o,y)-d(x,y)).$$
\end{defn}	

We define the Gromov hyperbolicity of arbitrary metric space $X$ which coincides with the thin triangles definition when $X$ is a geodesic metric space.
\begin{defn}
	 Let $\delta\ge 0$. A metric space $(X,d)$ is said to be $\delta$-hyperbolic if there exists a $\delta \geq 0$ such that $\forall\,x,y,z,p \in X$, the following inequality holds:
	$$(x|z)_p\,\geq\,\min\{(x|y)_p,(y|z)_p\}-\delta.$$
\end{defn}

\begin{defn}
	\begin{itemize}
		\item[(1)] We say that a sequence $\{x_n\}\subset X$ is a {\it Gromov sequence} if it goes to infinity in Gromov sense {\it i.e.,} if
		$$\lim_{m,n\to \infty}(x_n|x_m)_o= +\infty.$$
		\item[(2)] Let $\Gamma^G$ denote the set of Gromov sequence and  $\{x_n\},\{y_n\}\in \Gamma^G$. If $X$ is Gromov hyperbolic then the following relation 
		$$x_n\sim_s y_n \iff \lim_{n\to \infty}(x_n|y_n)_o= \infty$$
		is an equivalence relation on $\Gamma^G$. This definition is also irrelevant of chosen point $o$, since $|(x|y)_p-(x|y)_q|\le d(p,q)$.
		\item[(3)] We define the set
		$$\partial_{G}^{S}X:=\Gamma^G/\sim_s.$$
		as the {\it Gromov boundary}. 
	\end{itemize}
\end{defn}
	The set $\partial_{G}^{S}X$ admits a canonical bijection with $\partial_{G}X$, that is,
	$$\partial_{G}^{S}X\to \partial_{G}X \,\,\,\,\, \mbox{ by }\,\,\,\,\, [\gamma]\to [(\gamma(n))_{n\in \mathbb{N}}]_s.$$
	
\subsection*{Result} Let $(X,d)$ be Gromov hyperbolic. A sequence $(x_n)_{n\ge 1}$ in $X$ converges to a point $\xi\in \partial_{G}X$ if, and only if, $(x_n)_{n\ge 1}$ is a Gromov sequence and $[(x_n)_{n \ge 1}]=\xi$.\\

The following proposition says that two sequences converge to the same points of the Gromov boundary if, and only if, they are equivalent.
	\begin{prop}
		Let $x_n$, $y_n$ be two sequences such that
		$$x_n\to \xi\in\partial_{G}X,\,\,\,\,\,\, y_n\to \eta \in \partial_{G}X.$$
		Then $\xi=\eta$ if, and only if, $(x_n|y_n)_o\to \infty$.
	\end{prop}
	Moreover, the Gromov boundary $\partial_{G}X$ has a canonical quasiconformal structure in the following sense.
	
	\begin{thm}\label{QI extension}
		Let $X$ and $Y$ be proper geodesic Gromov hyperbolic space. If $f:X\to Y$ is quasi-isometry, then $f$ extends to a homeomorphism $\overline{f}:\partial_{G}X\to\partial_{G}Y$. Moreover, the boundary map is quasi-conformal.
	\end{thm}

	\subsection*{Visibility in Gromov hyperbolic spaces}
In view of the Gromov topology of $\overline{X}^G$ we have the following observation for geodesic rays and geodesic lines.
\begin{rem}\label{limit}
	For any geodesic ray $\gamma:[0,\infty)\to X$,
	$$\lim_{t\to\infty}\gamma(t)=\mbox{ End }(\gamma)=[\gamma]\in \partial_{G}X$$
	and for any geodesic line $\gamma:(-\infty,\infty)\to X$, both the following limits exist and 
	$$ \lim_{t\to-\infty}\gamma(t)\ne\lim_{t\to \infty}\gamma(t).$$
	For a geodesic line $\gamma$, we denote $\gamma(-\infty)=\lim_{t\to-\infty}\gamma(t)$ and $\gamma(\infty)=\lim_{t\to \infty}\gamma(t)$. 
\end{rem}
Following theorem tells us that any point in the Gromov hyperbolic space can be joined to any Gromov boundary point via geodesic ray.

\begin{thm}\cite[Chapter III.H.3, Lemma 3.1]{Bridson-book}
	For each $p\in X$ and $\xi\in \partial_{G}X$ there exists a geodesic ray $\gamma:[0,\infty)\to X$ such that $\mbox{ End }(\gamma)=\xi$.
\end{thm}
Gromov hyperbolic space has the following visibility property. 
\begin{thm}\cite[Chapter III.H.3, Lemma 3.2]{Bridson-book}\label{GV-I}
	Let $(X,d)$ be a proper, geodesic and Gromov hyperbolic metric space, then for each distinct points $\xi, \eta\in \partial_{G}X$,
	there exists a geodesic line $\gamma$ with 	$\gamma(-\infty)=\xi$ and $\gamma(\infty)=\eta$.
\end{thm}
Theorem \ref{GV-I} can be interpret in the following equivalent way. For the proof see Section \ref{Appendix}.

\begin{thm}\label{GV}
	Let $(X,d)$ be a proper, geodesic and Gromov hyperbolic metric space. If $\xi, \eta\in \partial_{G}X$ and $V_{\xi}$ and $V_{\eta}$ are neighborhoods in $\overline{X}^G$ of $\xi$ and $\eta$ respectively so that $\overline{V_{\xi}}\cap \overline{V_{\eta}}=\emptyset$ 
	then there exists a compact set $K\subset X$ such that:
if $\gamma:[0,T]\to X$ is a geodesics with $\gamma(0)\in V_{\xi}$ and $\gamma(T)\in V_{\eta}$, then $\gamma([0,T])\cup K\ne \phi$. 
\end{thm}

One can observe that the Theorem \ref{GV} is equivalent to the following.

\begin{rem}\label{EQGV}
	Let $(X,d)$ be a proper, geodesic and Gromov hyperbolic metric space.  If $\xi, \eta\in \partial_{G}X$, $\xi\ne \eta$ and $z_k$, $w_k$ be any sequence with  $z_k \stackrel{Gromov}\longrightarrow \xi$ and $w_k \stackrel{Gromov}\longrightarrow \eta$, then there exists a compact set $K\subset X$ with the following property: if $\gamma_k:[0,T_k]\to X$ be a sequence of geodesics with $\gamma_k(0)=z_k$ and $\gamma_k(T_k)=w_k$, then  $\gamma_k([0,T_k])\cap K\ne \phi$.
\end{rem}

	\subsection{Quasihyperbolic Domains}\label{QHG and GH} In this section we recall some classical domains for which we will study the visibility property.
	
	\subsection*{Uniform Domains}\label{uniform} Let $C\ge 1$. A domain $\Omega\subset\mathbb{R}^n$ is said to be $C$-uniform domain if for each pair of points $x,y\in \Omega$, there exists a rectifiable arc $\gamma$ in $\Omega$ joining $x$ and $y$ such that 
	\begin{itemize}
		\item[(i)] $\min\{l(\gamma[x,z]),l(\gamma[z,y])\}\le C\,\delta_{\Omega}(z)$, for all $z\in \gamma$, and
		\item[(ii)]	$l(\gamma)\le C|x-y|$,	
	\end{itemize}
	where $l(\gamma)$ denotes the arc length of $\gamma$ in $(\Omega,|.\,|)$, where $|.\,|$ is metric induced by norm, and $\gamma[x,z]$ denotes the part of $\gamma$ between $x$ and $z$. Also we say that curve $\gamma$ is a double $C$-cone arc. 
	
	\subsection*{Inner uniform domains}Let $C\ge 1$. A domain $\Omega\subset\mathbb{R}^n$ is said to be $C$-inner uniform in the norm metric if for each pair of points $x,y\in \Omega$, there exists a rectifiable arc $\gamma$ in $D$ joining $x$ to $y$ such that 
	\begin{itemize}
		\item[(i)] $\min\{l(\gamma[x,z]),l(\gamma[z,y])\}\le C\,\delta_{\Omega}(z)$, for all $z\in \gamma$,
		\item[(ii)]	$l(\gamma)\le C \lambda_{\Omega}(x,y)$.
	\end{itemize}
	
	\subsection*{John Domain}\label{John}	Let $C\ge 1$. A domain $\Omega\subset\mathbb{R}^n$ is said to be $C$-John domain with center $x_0\in \Omega$ if every point $x\in \Omega$ can be joined to $x_0$ by a rectifiable curve $\gamma$ in $\Omega$ which satisfies 
	\begin{equation}\label{John-2}
		l(\gamma[x,y])\le C\delta_{\Omega}(y)\,\,\, \mbox{ for all } y\in \gamma
	\end{equation}
	and we say that such a curve $\gamma$ is a $C$-cone arc. Note that from above definition it follows that $\Omega$ is bounded $(\Omega\subset B(x_0,C\delta_{\Omega}(x_0)))$. The following definition works well for unbounded domains.
	
	\begin{defn}
		Let $C\ge 1$. A domain $\Omega\subset\mathbb{R}^n$ is said to be $C$-John domain if for each pair of points $x,y\in \Omega$, there exists a rectifiable arc $\gamma$ in $\Omega$ joining $x$ and $y$ such that 
		$$\min\{l(\gamma[x,z]),l(\gamma[z,y])\}\le C\delta_{\Omega}(z), \mbox{ for all } z\in\gamma.$$
	\end{defn}

	\begin{rem}
		By \cite[Lemma 2.2]{Rasila-2022}, we have that if $\Omega$ is $C$-John domain with center $x_0\in \Omega$, the $\Omega$ is $C$-John. Moreover if $\Omega$ is $C$-John then $\Omega$ is $4C^2$-John with center $x_0\in \Omega$, where $x_0$ is such that $\delta_{\Omega}(x_0)=\max \delta_{\Omega}(x)$. Since we are considering bounded domains only so we will use both the definitions according to our convinience.
	\end{rem}

	\begin{rem} We mention some of the remarks which gives the connection between the aforesaid domains.
		
		\begin{itemize}
			\item[(i)] $C$-uniform implies inner $C$-uniform implies $C$-John
			\item[(ii)] $\mathbb{R}^2\setminus \{(x,0):x\ge 0\}$ is inner uniform but not unifrom.
			\item[(iii)] $\mathbb{R}^2\setminus \{(n,0):n\in \mathbb{N}\}$ is a John domain but not inner uniform.
		\end{itemize}
	\end{rem}

	\begin{defn}[\textbf{QHBC domains}]\label{QHBC-domains}
Following \cite{Gehring-1985} we say that a domain $\Omega\subset\mathbb{R}^n$ satisfies a quasihyperbolic boundary condition if for some $x_0$ there exists $\beta>0$  and $C_0=C_0(x_0)$ such that
\begin{equation}\label{QHBC}
	k_{\Omega}(x_0,x) \leq \frac{1}{\beta} \log \dfrac{\delta_{\Omega}(x_{0})}{\delta_{\Omega}(x)} + C_0\,\,\,
\end{equation}
for all $x\in \Omega$. 
	\end{defn}
	It easily follows that if $\Omega$ satisfies \eqref{QHBC} for some $x_0\in \Omega$, then $\Omega$ satisfies \eqref{QHBC} for any $x_1\in \Omega$ with same $\beta$ and possibly a different $C=C_0+k_{\Omega}(x_0,x_1)$.  Also, notice that, in view of \eqref{qh-eq-1}, it is necessary that $\beta\le 1$ . In short, we say that $\Omega$ is  $\beta$-QHBC domain. Every $\beta$-QHBC domain $\Omega$ is bounded (see \cite[Lemma 3.9]{Gehring-1985}). Infact, $\mbox{diam}(\Omega)\le 2\beta e^{C_0/\beta}$.
	
	\begin{rem}\label{John-QHBC}
		In view of \cite[Lemma 3.11]{Gehring-1985}, we have that every bounded $C$-John domain satisfies the following quasihyperbolic boundary conditions
		$$k_{\Omega}(x,x_{0}) \le C \log \dfrac{\delta_{\Omega}(x_{0})}{\delta_{\Omega}(x)} + A,\,\,\, \mbox{ for all } x \in \Omega,$$
		where $C\ge 1$ and $A = C + C\log(1+C)$. 
	\end{rem}
	
	\subsection{Gromov hyperbolicity of hyperbolic and quasihyperbolic metric}
	\begin{defn}
		Let $\delta\ge 0$. A domain $\Omega\subset\mathbb{R}^n$ is said to be $\delta$-hyperbolic if $(\Omega,k_\Omega)$ is $\delta$-hyperbolic. 
	\end{defn}
	
	\begin{defn}
		A domain $\Omega\subset\mathbb{R}^n$ is said to be Gromov hyperbolic if $(\Omega,k_\Omega)$ is $\delta$-hyperbolic for some $\delta\ge 0$.
	\end{defn}
	It is known that the following domains in $\mathbb{R}^n$ are Gromov hyperbolic
	\begin{itemize}
		\item[(1)] All $c$- uniform and inner $c$- uniform domains are Gromov $\delta=\delta(c)$ hyperbolic;
		\item[(2)] bounded convex domains in $\mathbb{R}^n$;
		\item[(3)]Quasiconformal image of Gromov hyperbolic domains. In particular, proper simply connected domains in $\mathbb{C}$ are Gromov hyperbolic.
	\end{itemize}
	 $\mathbb{R}^2\setminus \{(n,0):n\in \mathbb{N}\}$ is a John domain which is not Gromov hyperbolic.
	
	\subsection*{Gehring-Hayman condition}
	Let $\Omega \subset \mathbb{R}^n$ be a proper subdomain. $\Omega$ satisfies  a Gehring-Hayman condition if there is a constant $C_{gh} \geq 1$ such that for each pair of points $x,y \in \Omega$ and for each quasi-hyperbolic geodesic $\gamma$ joining $x$ and $y$, it holds that-
	$$l(\gamma[x,y]) \leq C_{gh}l(\sigma[x,y]),$$
	where $\sigma$ is any other curve joining $x$ and $y$ in $\Omega$. In other words, it says that quasihyperbolic geodesics are essentially the shortest curves in $\Omega$.
	
	\subsection*{Ball separation condition}	Let $\Omega \subset \mathbb{R}^n$ be a proper subdomain. $\Omega$ satisfies  a ball separation condition if there exists a constant $C_{bs} \geq 1$ such that for each pair of points $x$ and $y$, for
	each quasihyperbolic geodesic $\gamma$ joining $x$ and $y$, for every $z \in \gamma[x, y]$, and for every curve $\sigma$ joining $x$ and $y$ it holds that-
	$$B(z, C_{bs}d(z,\partial \Omega)) \cap \sigma[x,y] \neq \phi$$
	Bonk {\it et. al.} \cite[Proposition 7.12, 7.14]{Koskela-2001} have proved that Gehring-Hayman condition and Ball separation condition are necessary for $(\Omega,k_{\Omega})$ to be Gromov hyperbolic. Later, in 2003, Balogh and Buckley \cite[Theorem 0.1]{Balogh-2002} proved that both the conditions together are sufficient also. Thus, combining these two we have the following result.
	\begin{thm}{\cite[Proposition 7.12, 7.14]{Koskela-2001}, \cite[Theorem 0.1]{Balogh-2002}}
		Let $\Omega \subset \mathbb{R}^n$ be a proper subdomain. Then $(\Omega,k_{\Omega})$ is Gromov hyperbolic if, and only if, $\Omega$ satisfies both a Gehring-Hayman condition and ball separation condition.
	\end{thm}	
\section{Equivalence of Gromov boundary and Euclidean boundary}\label{GH and visibility}
We first give the proof of Theorem \ref{gromov product visibility}. For this we need the following lemma which gives an equivalent visibility criteria.
\begin{lem}\label{iff}
	Let $\Omega$ be a bounded domain in $\mathbb{R}^n$. Then $\Omega$ is a QH-visibility domain if, and only if, the following holds:\\
	If $p,q\in \partial_{Euc}\Omega$, $p\ne q$, and $\{x_k\}, \{y_k\}$ be sequence in $\Omega$ with $x_k\to p$, $y_k\to q$, then for every qh-geodesic $\gamma_k:[a,b]\to \Omega$ with $\gamma_k(a)=x_k$ and $\gamma_k(b)=y_k$, there exists a positive constant $M>0$  such that
	$$\sup_{k\ge 1}{k_{\Omega}}(x_0,\gamma_k)\le M$$
	for every $x_0\in \Omega$.
\end{lem}
\begin{pf}
	Let $\Omega$ be a bounded domain in $\mathbb{R}^n$. Fix $x_0\in \Omega$ and suppose $\Omega$ is a QH-visibility domain, then there exists a compact set $K\subset\Omega$ such that $\gamma_k$ intersects with $K$ for sufficiently large $k$. Let $z_k$ be the intersection points. Then for sufficiently large $k$, we have the following
	$$k_{\Omega}(x_0,z_k)\le k_{\Omega}(x_0,z),\,\,\,\,\, z\in \gamma_k$$
	That is $k_{\Omega}(x_0,\gamma_k)=k_{\Omega}(x_0,z_k)$ for suffieiently large $k$, and since for each $k$, $z_k$ lies in a compact set, we have $$\sup_{k\ge 1}{k_{\Omega}}(x_0,\gamma_k)\le M.$$
	Conversly, Let $p,q\in \partial_{Euc}\Omega$, $p\ne q$, and $\{x_k\}, \{y_k\}$ be sequences in $\Omega$ with $x_k\to p$, $y_k\to q$, then for every qh-geodesic $\gamma_k:[a,b]\to \Omega$ with $\gamma_k(a)=x_k$ and $\gamma_k(b)=y_k$, we have 
	$$\sup_{k\ge 1}{k_{\Omega}}(x_0,\gamma_k)\le M$$
	for some $M>0$. Since $\gamma_k$ is compact in $(\Omega,k_{\Omega})$, therefore there exits $z_k\in \gamma_k$ such that 
	$$k_{\Omega}(x_0,\gamma_k)=k_{\Omega}(x_0,z_k).$$ 
	By the assumption $\sup_{k\ge 1}k_{\Omega}(x_0,z_k)<M$, Lemma  gives that the set $\{z_k\}$ is compactly contained in $\Omega$. Therefore, there exists a compact set $K$ with $\{z_k\}\subset K$, and hence $\gamma_k$ intersects with $K$, which gives that $\Omega$ is a QH-visibility domain. This completes the proof.
\end{pf}

	\begin{pf}[\textbf{Proof of Theorem \ref{gromov product visibility}}]
	$(1)\implies (2)$: Let $\Omega$ be a QH-visibility domain and $z_k$, $w_k$ be any sequence in $\Omega$ with $z_k\to p$ and $w_k\to q$, where $p,q\in \partial_{Euc}\Omega$, $p\ne q$. Let $\gamma_k$ be qh-geodesics joining $z_k$ and $w_k$. Then by the visibility property there exists a compact set $K$ and $k_0\in \mathbb{N}$ such that $\gamma_k$ intersects with $K$ for all $k\ge k_0$. Let $\theta_k$ be the inetrsection points of $\gamma_k$ and $K$. Fix any point $o\in \Omega$. then for all $k\ge k_0$, we have
	\begin{eqnarray*}
		k_{\Omega}(z_k,w_k)&=&k_{\Omega}(z_k,\theta_k)+k_{\Omega}(\theta_k, w_k)\\ \nonumber
		&\ge&k_{\Omega}(z_k,o)+k_{\Omega}(o,w_k)-2k_{\Omega}(o,\theta_k)\\ \nonumber
		&\ge&k_{\Omega}(z_k,o)+k_{\Omega}(o,w_k)-2\max_{\theta\in K}k_{\Omega}(o,\theta)
	\end{eqnarray*}
	which gives that 
	$$\limsup_{k\to \infty}(z_k|w_k)_o=\lim_{k\to \infty}\frac{1}{2}\left(k_{\Omega}(z_k,o)+k_{\Omega}(o,w_k)- 	k_{\Omega}(z_k,w_k)\right)\le\max_{\theta\in K}k_{\Omega}(o,\theta)< \infty.$$ 
	$(2)\implies(1)$: Assume by contradiction that $\Omega$ is not a QH-visibility domain. Then by Lemma \ref{iff} there is a pair $\{p,q\}\subset\partial_{Euc}\Omega$, a sequences $z_k\to p$ and $w_k\to q$ and a sequence of qh-geodesics $\gamma_{k}$ joining $z_k$ and $w_k$ such that 
	$$\sup_{k\ge 1}k_{\Omega}(o,\gamma_{k})=\infty$$
	Hence
	\begin{equation}\label{*}
		\limsup_{k\to \infty}k_{\Omega}(o,\gamma_{k})=\limsup_{k\to \infty}k_{\Omega}(o,\theta_{k})=\infty
	\end{equation}
	where $\theta_k\in \gamma_{k}$. Now by the completeness property of quasihyperbolic metric it follows that $\theta_k\to p^*\in\partial_{Euc}\Omega$. We can always choose some $z_k$ at the place of $\theta_{k}$ so that $p^*$ will be different from $p$ and $q$. It is easy to see that
	\begin{eqnarray*}
		k_{\Omega}(z_k,o)+k_{\Omega}(o,w_k)&\ge& k_{\Omega}(z_k,w_k)=k_{\Omega}(z_k,\theta_k)+k_{\Omega}(\theta_k, w_k)\\ \nonumber
		&=&-2(z_k|\theta_k)_o-2(\theta_{k}|w_k)_o+k_{\Omega}(z_k,o)+k_{\Omega}(o,w_k)+2k_{\Omega}(o,\theta_k)\\ \nonumber
		&\ge&k_{\Omega}(z_k,o)+k_{\Omega}(o,w_k)+2k_{\Omega}(o,\theta_k)-C
	\end{eqnarray*}
	This gives that $k_{\Omega}(o,\theta_k)<C$, which contradicts \eqref{*}. This completes the proof.
\end{pf}

We are now going to prove Theorem \ref{Main-thm-GB-EB} and Theorem \ref{no loop}. We begin with some definitions.
	
	\begin{defn}
	Fix $x_0\in \Omega$ and	let $\gamma_{k}:[0,T_k]\to(\Omega,k_{\Omega})$ be a sequence of qh-geodesics issuing from $x_0$. We say that $\gamma_{k}$ converges uniformly on compacta to a geodesic ray $\gamma:[0,T]\to (\Omega,k_{\Omega})$, if $T_k\to \infty$ and for any fixed $T>0$, the following holds:\\
	for every $\epsilon>0$, there exists a natural number $N$ such that for all $n\ge N$, we have
	$$\sup_{t\in [0,T]} k_{\Omega}(\gamma_{k}(t), \gamma(t))<\epsilon.$$ 
	\end{defn}	
	
	\begin{defn}
		Let $\Omega$ be a bounded domain in $\mathbb{R}^n$ and $\gamma:[0,\infty)\to \Omega$ be a qh-geodesic ray. We denote the cluster set of $\gamma$ at $\infty$ by $\Gamma^{\infty}_{\gamma}$ and is defined as 
		$$\Gamma^{\infty}_{\gamma}=\left\{x\in \overline{\Omega}^{Euc}:\mbox{there exists a sequence } t_k\to \infty \mbox{ such that } \lim_{k\to \infty}\gamma(t_k)=x\right\}$$
		similarly, for qh-geodesic line we can define the cluster set at $-\infty$ as 
		$$\Gamma^{-\infty}_{\gamma}=\left\{x\in \overline{\Omega}^{Euc}:\mbox{there exists a sequence } t_k\to -\infty \mbox{ such that } \lim_{k\to -\infty}\gamma(t_k)=x\right\}$$
	\end{defn}
	
\begin{rem}\label{5.2}
Let $\Omega$ be a bouded domain in $\mathbb{R}^n$, and $\gamma$ be a qh-geodesic in $\Omega$. Then for any sequence $t_k\to \infty$, we have $|t_k|=k_{\Omega}(\gamma(0),\gamma(t_k)) \to \infty$ as $k\to \infty$. By Result 1, it is clear that all the limit points of $\{\gamma(t_k)\}$ belongs to $\partial_{Euc}\Omega$. In particular, $\Gamma^{\infty}_{\gamma}\subset \partial_{Euc}\Omega$.
\end{rem}
	
	\begin{defn}
		Let $\Omega \in \mathbb{R}^n$ be a bounded domain and $X$ be a compactification of $\Omega$ (for instance, $\overline{\Omega}^{Euc}$ and $\overline{\Omega}^G$ are compactification of $\Omega$). We say that a qh-geodesic line $\gamma:(-\infty,\infty)\to \Omega$ is a qh-geodesic loop in $X$ if $$\Gamma^{\infty}_{\gamma}=\Gamma^{-\infty}_{\gamma}$$
	\end{defn}
	
	\begin{defn}
		If there is no qh-geodesic line in $\Omega$ which is a qh-geodesic loop in $X$, then we say that $\Omega$ has no geodesic loop in $X$.
	\end{defn}
	In view of Remark \ref{limit}, $\Omega$ has no geodesic loop in $\overline{\Omega}^G$. On the other hand, $\Omega$ can have geodesic loop in $\overline{\Omega}^{Euc}$. See the following example,
	\begin{example}
		Let $\Omega=\mathbb{D}\setminus\{(-1,0]\}$. Then in view of Riemann mapping theorem, consider the conformal (analytic, one-to-one) map
		$f:\mathbb{D}\to \Omega$. Let $p,q\in \partial_{Euc}\mathbb{D}$, $p\ne q$, and a geodesic line $\gamma$ joining $p$ and $q$. Since every Riemann map has continuous extension on the boundary, we can assume that $f(p)=f(q)=-1/2$. Hence, $f(\gamma)$ is a geodesic line in $\Omega$ which is a geodesic loop in $\overline{\Omega}^{Euc}$.
	\end{example}

In the next lemma, we show that in the case of QH-visibility domain, for every sequence $\{t_k\}$ with $t_k\to\infty$ as $k\to \infty$, the limit $\lim_{k\rightarrow \infty}\gamma(t_k)$ exist, and hence by Remark \ref{5.2} it belongs to $\partial_{Euc}\Omega$.

\begin{lem}
Let $\Omega$ be a bounded domain in $\mathbb{R}^n$, and $\gamma$ be a qh-geodesic ray in $\Omega$. If $\Omega$ is a QH-visibility domain, then for any sequence $t_k\to \infty$ as $k\to \infty$, $\lim_{k\to \infty}\gamma(t_k)$ exists in $\partial_{Euc}\Omega$. 
\end{lem}
	
	\begin{pf}
	Since $\overline{\Omega}^{Euc}$ is compact, $\{\gamma(t_k)\}$ has a limit point, and by Remark \ref{5.2}, all its limit point belongs to $\partial_{Euc}\Omega$. As $\{\gamma(t_k)\}$ is a bounded infinite sequence, to show that it converges, we shall show that it has a unique limit point. Suppose for contradiction that $\gamma(t_k)$ has two distinct limit points $p,q\in \partial_{Euc}\Omega$ {\it i.e,} there are two subsequence of $\gamma(t_k)$, say, $\gamma(t_{n_{k}})$ and $\gamma(t_{s_{k}})$ such that 
		$$\lim\limits_{k\to \infty}\gamma(t_{n_{k}})=p\ne q= \lim\limits_{k\to \infty}\gamma(t_{s_{k}}).$$
		Upto subsequence we can assume $t_{n_{k}}<t_{s_{k}}<t_{n_{k+1}}$ for all $k$. Consider the qh-geodesics $\gamma_k=\gamma|_{[t_{n_{k}},t_{s_{k}}]}$. Since $\Omega$ is a QH-visibility domain, therefore there exists a compact set $K\subset \Omega$ such that for sufficiently large $k$, there exists $t'_{n_{k}}\in (t_{n_{k}},t_{s_{k}})$ such that $\gamma(t'_{n_{k}})\in K$. But then $\Gamma^{\infty}_{\gamma}\cap K\ne \emptyset$, which is a contradiction.
	\end{pf}

In the next lemma, we show that in the case of QH-visibility domain, for every sequence $\{t_k\}$ with $t_k\to\infty$, the limit $\lim_{k\rightarrow \infty}\gamma(t_k)$ is same. We give a name to this phenomenon: We say that a qh-geodesic ray $\gamma:[0,\infty)\to \Omega$  lands at a point $p\in \partial_{Euc}\Omega$ if $\Gamma^{\infty}_{\gamma}=\{p\}$. 

	\begin{lem}\label{lands}
		Let $\Omega$ be a QH-visibility domain. Then any qh-geodesic ray $\gamma$ {\it lands at a point} $p\in  \partial_{Euc}\Omega$ {\it i.e.,} $\lim\limits_{t\to \infty}\gamma(t)$ exists in $\partial_{Euc}\Omega$.
	\end{lem}
	
	\begin{pf}
		Suppose there are two distinct points $p,q\in \Gamma^{\infty}_{\gamma}$, {\it i.e.,} there exists sequences $s_k\to \infty$ and $t_k\to \infty$ such that $\lim_{k\to \infty}\gamma(s_k)=p$ and $\lim_{k\to \infty}\gamma(t_k)=q$. Upto subsequence we can assume $s_k<t_k<s_{k+1}$ for all $k$. Consider the qh-geodesics $\gamma_k=\gamma|_{[s_k,t_k]}$, then by the visibility property, there exists a compact set $K\subset \Omega$ such that for sufficiently large $k$, there exists $s'_k\in (s_k,t_k)$ such that $\gamma(s'_k)\in K$. But then $\Gamma^{\infty}_{\gamma}\cap K\ne \phi$, which is a contradiction.
	\end{pf}
	
Next we show that in QH-visibility domain, equivalent qh-geodesic rays lands at the same point. This will gives us that the map is well defined.
	
	\begin{lem}\label{well-defined}
		Let $\Omega$ be QH-visibility domain in $\mathbb{R}^n$. Fix $x_0\in \Omega$ and let $\gamma$ and $\sigma$ be any two qh-geodesic rays with $\gamma(0)=x_0=\sigma(0)$ such that $\gamma \sim \sigma$. Then $\gamma$ and $\sigma$ lands at a same point in $\partial_{Euc}\Omega$.
	\end{lem}
	
	\begin{pf}
		Let $\Omega$ be a visibility domain. In view of lemma \ref{lands}, by contradiction suppose that $\gamma$ and $\sigma$ lands at two distinct point $p,q\in\partial_{Euc}\Omega$. For any sequence $\{s_k\}$ with $s_k\to \infty$ write $x_k=\gamma(s_k)$ and $y_k=\sigma(s_k)$, it is clear that that $x_k\to p$ and $y_k\to q$ in the Euclidean metric. In view of Observation \ref{observation}, we have that $\lim_{k\to \infty}k_{\Omega}(z_k,w_k)=\infty$, which is a contradiction to the fact that $\gamma\sim \sigma$.
	\end{pf}

Now we are ready to define the extension map. Fix a point $x_0\in \Omega$, and let $\gamma$ be any representative of $\xi\in \partial_{G}\Omega$. In view of Lemma \ref{lands}, define the map $\Phi:\overline{\Omega}^G\to \overline{\Omega}^{Euc}$ as
\[\Phi(z)=\begin{cases}
	z&\mbox{ if } x\in \Omega \\[3mm]
	p&\mbox{ if } \xi\in\partial_{G}\Omega
\end{cases}
\]
In view of Lemma \ref{well-defined}, the map $\Phi$ is well-defined.	The next lemma will help us to prove that $\Phi$ is continuous.

\begin{lem}\label{continuous}
	Let $\Omega$ be a bounded domain. Suppose $\Omega$ is a QH-visibility domain.
	\begin{itemize}
		\item[(i)] If $\gamma_k:[0,\infty)\to (\Omega,k_{\Omega})$ be a sequence of qh-geodesic rays converges uniformly on compacta to a geodesic ray $\gamma:[0,\infty)\to (\Omega,k_{\Omega})$, then
		$$\lim\limits_{t\to \infty}\gamma(t)=\lim\limits_{k\to \infty}\lim\limits_{t\to\infty}\gamma_k(t).$$
		\item[(ii)]  If $\gamma_k:[0,T_k]\to (\Omega,k_{\Omega})$ be a sequence of qh-geodesics converges uniformly on compacta to a geodesic ray $\gamma:[0,\infty)\to (\Omega,k_{\Omega})$, then
		$$\lim\limits_{t\to \infty}\gamma(t)=\lim\limits_{k\to \infty}\gamma_k(T_K).$$
	\end{itemize}
\end{lem}

\begin{pf}
	Since $\overline{\Omega}^{Euc}$ is compact it is enough to consider the case when the limit
	$$\lim\limits_{k\to \infty}\lim\limits_{t\to T_k}\gamma_k(t)\,\,\, \mbox{exists}.$$ 
\begin{pf}[Proof of (i)]
By Lemma \ref{lands}, for each $k$, we have that $\lim_{t\to \infty}\gamma_{k}(t)=p_k\in \partial_{Euc}\Omega$, and $\lim_{t\to \infty}\gamma(t)=q$. Let $\lim_{k\to \infty}p_k=p\in \partial_{Euc}\Omega$. Our aim is to show that $p=q$. If not, then as $\lim_{k\to \infty}\gamma_{k}(t)=\gamma(t)$ for every $t$, we can pick $a_k\to \infty$, $b_k\to \infty$ such that $\gamma_{k}(a_k)\to p$ and $\gamma_{k}(b_k)\to q$. Upto a subsequence we can assume $a_k<b_k$. Then by visibility property, there exists a comapct set $K\subset \Omega$ such that $\gamma_k|_{[a_k,b_k]}$ intersects with $K$ for sufficiently large $k$, say $\gamma_k(c_k)\in K$ for some $c_k\in [a_k,b_k]$. Therefore, $$c_k=k_{\Omega}(x_0=\gamma_k(0),\gamma_k(c_k))\le\max_{x\in K}k_{\Omega}(x_0,x).$$
which gives that $c_k$ does not tend to $\infty$, a contradiction. Hence, $p=q$. This completes the proof.
\end{pf}

\begin{pf}[Proof of (ii)]	
Let $\gamma_{k}(T_k)=x_k\in \Omega$. Sicne each $\gamma_{k}$ is a qh-geodesic, the sequence $\{x_k\}$ converges to a point in $p\in\partial_{Euc}\Omega$. Further, by Lemma \ref{lands}, $\lim_{t\to\infty}\gamma(t)=q\in \partial_{Euc}\Omega$. Our aim is to show that $p=q$. If not, then as $\lim_{k\to \infty}\gamma_{k}(t)=\gamma(t)$ for every $t$, we can pick $a_k\to \infty$, $b_k\to \infty$ such that $\gamma_{k}(a_k)\to p$ and $\gamma_{k}(b_k)\to q$. Upto a subsequence we can assume $a_k<b_k$. Then by visibility property, there exists a comapct set $K\subset \Omega$ such that $\gamma_k|_{[a_k,b_k]}$ intersects with $K$ for sufficiently large $k$, say $\gamma_k(c_k)\in K$ for some $c_k\in [a_k,b_k]$. Therefore, $$c_k=k_{\Omega}(x_0=\gamma_k(0),\gamma_k(c_k))\le\max_{x\in K}k_{\Omega}(x_0,x).$$
which gives that $c_k$ does not tend to $\infty$, a contradiction. Hence, $p=q$. This completes the proof.
\end{pf}

Combining the Proof of (i) and Proof of (ii), we have the proof of the Lemma.	
\end{pf}

Following lemma will give us that $\Phi$ is onto.
	
	\begin{lem}\label{onto}
		Suppose $\Omega$ is a QH-visibility domain. Fix any point $p\in \partial_{Euc}\Omega$ and $x_0\in \Omega$, then there exists a qh-geodesic ray $\gamma$ issuing from $x_0$ which lands at $p$.
	\end{lem}
	
	\begin{pf}
		Let $p\in \partial_{Euc}\Omega$, there is a sequence $\{x_k\} \subset\Omega$ which converges to point $p$. Let $\gamma_k$ be a sequence of qh-geodesics joining $x_0$ and $x_k$. As $(\Omega,k_{\Omega})$ is a proper metric space, by Arzel\'a-Ascoli's theorem, and a diagonal argument, upto a subsequences, we can assume that $\{\gamma_k\}$ converges uniformly on compacta to a qh-geodesic ray $\gamma$. Then by Lemma \ref{continuous} (ii), $\gamma$ lands at $p\in \partial_{Euc}\Omega$. This completes the proof.
	\end{pf}	

\begin{rem}
There is another topological way to prove that $\Phi$ is onto: Since $\overline{\Omega}^G$ and $\overline{\Omega}^{Euc}$ is compact and $\Phi$ is continuous with $\Phi(\Omega)=\Omega$ which is dense in $\overline{\Omega}^{Euc}$, theorefore $\Phi$ is onto.
\end{rem}
	\addtocontents{toc}{\protect\setcounter{tocdepth}{1}}
	\begin{pf}[\textbf{Proof of Theorem \ref{Main-thm-GB-EB}}]
		Suppose the identity map extends as a continuous surjective map $\widehat{id}:\overline{\Omega}^G \to \overline{\Omega}^{Euc}$. It is clear that $\widehat{id}(\partial_{G}\Omega)=\partial_{Euc}\Omega$. Let $p,q\in \partial_{Euc}\Omega$, $p\ne q$. Consider the two subsets $L_p=\widehat{id}^{-1}(p)$, $L_q=\widehat{id}^{-1}(q)$ of $\partial_{G}\Omega$. Clearly both are compact and $L_p\cap L_q=\emptyset$. Since $(\Omega,k_{\Omega})$ is Gromov hyperbolic, by Theorem \ref{GV}, it is easy to prove that there exist $\overline{\Omega}^G$-open subset $V_p$ and $V_q$ of $p$ and $q$ respectively and  compact set $K_{p,q}\subset \Omega$ such that $L_p\subset V_p$, $L_q\subset V_q$ and $\overline{V_p}\cap \overline{V_q}=\emptyset$ with the property that if $\gamma:[0,T]\to \Omega$ is any quasihyperbolic geodesic with 
		$\gamma(0)\in V_p\cap \Omega$ and $\gamma(T)\in V_q\cap \Omega$, then $\gamma([0,T])\cap K_{p,q}\ne \phi$.
		\par Now, suppose $x_k$ and $y_k$ are sequences in $\Omega$ converging in the Euclidean topology to $p$ and $q$ respectively. Then there exists $k_0\in \mathbb{N}$ such that $x_k\in V_p\cap \Omega$ and $y_k\in V_q\cap \Omega$, for all $k\ge k_0$. Thus, for $k\ge k_0$ any qh-geodesic $\gamma_k$ joining $x_k$ and $y_k$ intersects with $K_{p,q}$, and hence $\Omega$ is a QH-visibility domain.\\
		
		Conversly, assume that $\Omega$ is a QH-visibility domain. To show that $\Phi$ is continuous we need to consider following two cases
		\begin{itemize}
			\item[Case I:] If $\{z_k\}$ is a sequence in $\Omega$ such that $z_k \stackrel{Gromov}\longrightarrow \xi\in \partial_{G}\Omega$: In this case, for every qh-geodesic $\gamma_k:[0,R_k]\to \Omega$ with $\gamma_k(0)=z_0$ and $\gamma_k(R_k)=z_k$, its every subsequence has a subsequence which converges uniformly on compacta to a geodesic ray $\gamma\in \xi$, which in view of Lemma \ref{lands} lands at $p$. Thus, by Lemma \ref{continuous} (ii), we have $\Phi(z_k)=z_k\longrightarrow  p=\Phi(\xi)$.
			\item[Case II.] If $\{\xi_k\}$ is a sequence in $\partial_{G}\Omega$ such that $\xi_k \stackrel{Gromov}\longrightarrow \xi\in \partial_{G}\Omega$:\\
			Let $\gamma_k$ be the representative from $\xi_k$. Then $\xi_k \stackrel{Gromov}\longrightarrow \xi$ gives that every subsequence of $\gamma_k$ has a subsequence which converges uniformly on compacta to a geodesic ray $\gamma\in \xi$. Thus, by Lemma \ref{continuous} (i), we have that $\Phi(\xi_k)\stackrel{Euc}\longrightarrow \Phi(\xi)$. 
		\end{itemize}
		Hence in view of Case I and Case II, $\Phi$ is continuous. This completes the proof.
	\end{pf}
	
	\begin{pf}[\textbf{Proof of Theorem \ref{no loop}}]
		Since $\Omega$ has no geodesic loop in $\overline{\Omega}^G$ and $\Phi$ is a homeomorphism, it follows that $\Omega$ has no geodesic loop in $\overline{\Omega}^{Euc}$.
		\par Conversly, suppose $\Omega$ has no geodesic loop in $\overline{\Omega}^{Euc}$. Since $\overline{\Omega}^G$ and $\overline{\Omega}^{Euc}$ are Hausdorff and compact, therefore if we show that $\Phi$ is injective it will imply that $\Phi$ is a homeomorphism. Therefore, it is sufficient to prove that if $\Omega$ has no geodesic loop in $\overline{\Omega}^{Euc}$, then $\Phi$ is injective. This is equivalent to proving that if $\Phi$ is not injective, then $\Omega$ has a geodesic loop in $\overline{\Omega}^{Euc}$. Therefore, assume that $\Phi$ is not injective {\it i.e.,} $\Phi(\xi)=\Phi(\theta)$, for some $\xi,\theta\in\partial_{G}\Omega$, $\xi\ne \theta$. Let $\gamma\in \xi$ and $\sigma\in \theta$ be two representatives; $\gamma$ and $\sigma$ are qh-geodesic rays in $\Omega$ such that both $\gamma$ and $\sigma$ 
		\begin{itemize}
			\item lands at $p\in \partial_{Euc}\Omega$
			\item are not equivalent. In other words,
			$$\sup_{t\in[0, \infty)}k_{\Omega}(\gamma(t),\sigma(t))=\infty.$$
		\end{itemize} Therefore, there exists a sequence $t_k\to \infty$ such that
		$$\lim_{k\to \infty}k_{\Omega}(\gamma(t_k),\sigma(t_k))=\infty$$
		Now let $\eta_k:[0,R_k]\to \Omega$ be a sequence of qh-geodesics joining $\gamma(t_k)$ and $\sigma(t_k)$. Consider the geodesic triangle with sides $\gamma|_{[0,t_k]}$, $\eta_k$ and $\sigma|_{[0,t_k]}$.  Since $(\Omega,k_{\Omega})$ is Gromov hyperbolic, this geodesic triangle is $\delta$-thin, for some $\delta<\infty$, {\it i.e.,}
		\begin{equation}\label{0}
			\forall k\in \mathbb{N},\,\, \forall m\le k,\, \min\{k_{\Omega}(\gamma(t_m),\sigma([0,t_k])), k_{\Omega}(\gamma(t_m),\eta_k([0,R_k]))\}\le \delta
		\end{equation}
		It is easy to prove that
		\begin{equation}\label{0a}
			\lim\limits_{k\to \infty}k_{\Omega}(\gamma(t_k),\mbox{range}(\sigma))=\infty.
		\end{equation}
		By \eqref{0a}, choose $k_0$ such that
		\begin{equation}\label{0b}
			\forall k\ge k_0,\,\, k_{\Omega}(\gamma(t_k),\mbox{range}(\sigma))>\delta.
		\end{equation}	
		Using \eqref{0} and \eqref{0b}, we obtain
		$$\forall k\ge k_0,\,\, k_{\Omega}(\gamma(t_{k_0}),\eta_k([0,R_k]))\le\delta$$
		Consider the following compact set $K$ of $\Omega$
		$$K:=\{z\in \Omega:k_{\Omega}(z,\gamma(t_{k_0}))\le \delta\}.$$
			Then for every $k$, $\eta_k([0,R_k])\cap K\ne \phi$. We can parametrize $\eta_k$ in such a way that $\eta_k:[-s_k,t_k]\to \Omega$ such that $\eta_k(0)=z_0$. Then, by the visibility property of Gromov hyperbolic spaces (see \cite[Proposition 4.4 (2)]{O'Neill-1973}), we get that upto subsequence $\eta_k$ converges to a geodesic line $\eta:(-\infty,\infty)\to \Omega$ joining $\xi$ and $\theta$. Since $\Phi(\xi)=\Phi(\theta)$, therefore $\gamma$ and $\sigma$ lands at a same point point $p\in \partial_{Euc}\Omega$ {\it i.e.,} for any sequences $t_k\to \infty$ and $s_k\to\infty$, we have
		$$\lim_{k\to \infty}\gamma(t_k)=p \mbox{ and } \lim_{k\to \infty}\sigma(s_k)=p.$$
		Therefore, we have
		$$\lim_{t\to \infty}\eta(t)=p \mbox{ and } \lim_{t\to-\infty}\eta(t)=p.$$
		Therefore, $\eta$ is a geodesic line in $\Omega$ which is a geodesic loop in $\overline{\Omega}^{Euc}$. This completes the proof.
	\end{pf}

We will now provide the proof of Proposition \ref{Proposition-1} and Proposition \ref{Proposition-2}.

	\begin{pf}[\textbf{Proof of Proposition \ref{Proposition-1}}]
	By contradiction, suppose $\Omega$ has geodesic loop in $\overline{\Omega}^{Euc}$, {\it i.e}, there exists a geodesic line $\gamma:(-\infty,\infty)\to \Omega$ such that
	$$\lim_{t\to \infty}\gamma(t)=\lim_{t\to-\infty}\gamma(t)=p\in\partial_{Euc}\Omega.$$
	Let $x_k=\gamma(t_k)$ and $y_k=\gamma(s_k)$, where $s_k\to -\infty, t_k\to \infty$, then $$\lim_{k\to \infty}x_k=\lim_{k\to \infty}y_k=p\in\partial_{Euc}\Omega,$$
	and we can parametrize $\gamma$ in such a way that $\gamma(0)=x_0$. Now the sequence $\gamma_n$ converges uniformly on every compact subset to a geodesic line $\sigma$. In this case 
	$$\lim_{k\to \infty}k_{\Omega}(x_0,\gamma_k)\le k_{\Omega}(x_0,\sigma)\le M,$$
	which is a contradiction to the fact that $\Omega$ is quasihyperbolically well behaved. 
\end{pf}

\begin{pf}[\textbf{Proof of Proposition \ref{Proposition-2}}]
		Let $\Omega$ be a visibility domain. Suppose $\Omega$ has no geodesic loop in $\overline{\Omega}^{Euc}$. Let $\{x_k\}$, $\{y_k\}$ be two sequences in $\Omega$ with $$\lim_{k\to \infty}x_k=\lim_{k\to \infty}y_k=p\in\partial_{Euc}\Omega,$$
	and $\gamma_k:[a_n,b_n]\to \Omega$ be a sequence of qh-geodesic joining $x_k$ and $y_k$. We can parametrize $\gamma_k:[-s_k,t_k]\to \Omega$ such that $\gamma_k(-s_k)=x_k$ and $\gamma_k(t_k)=y_k$. Then $\gamma_k$ converges to a geodesic line $\gamma:(-\infty,\infty)\to \Omega$ such that 
	$$\lim_{t\to \infty}\gamma(t)=\lim_{t\to-\infty}\gamma(t)=p\in\partial_{Euc}\Omega$$
	which is a contradiction to the fact that $\Omega$ has no geodesic loop in $\overline{\Omega}^{Euc}$. The converse follows from Proposition \ref{Proposition-1}. This completes the proof. 
\end{pf}

\section{QH-visibility of some classical domains}\label{Visibility of domains}
In this section we study the visibility of some classical domains introduced in subsection \ref{QHG and GH}. More precisely, we provide the proof of Theorem \ref{cone arc-visibility} and Theorem \ref{visibility criteria}. 
	
\subsection{Cone-arc condition implies visibility} 
\addtocontents{toc}{\protect\setcounter{tocdepth}{1}}
In this subsection, we provide the proof of Theorem \ref{cone arc-visibility}.

\begin{pf}[\textbf{Proof of Theorem \ref{cone arc-visibility}}]
		Suppose $\Omega$ is not a QH-visibility domain i.e. there does not exist any compact set with the desired property. Therefore there exist distinct points $p,q\in \partial_{Euc} \Omega$, sequences $\{x_{k}\}, \{y_{k}\} \subset \Omega$ with $x_{k} \rightarrow p$ and $y_{k} \rightarrow q$, and a sequence of quasihyperbolic geodesics $\gamma_{k} :[a_{k},b_{k}] \rightarrow \Omega$ with $\gamma_{k}(a_{k}) = x_{k}$ and  $\gamma_{k}(b_{k}) = y_{k}$, such that
	\begin{equation}\label{1}
		\max_{{}t\in[ a_{k},{b_{k}]}} \delta_{\Omega} (\gamma_{k}(t)) \rightarrow 0   , \mbox{ as } k \rightarrow \infty
	\end{equation}
	Then by suitable parametrization, we can assume $a_k\le 0\le b_k$ for all $k\in \mathbb{N}$ and 
	$$\delta_{\Omega}(\gamma_{k}(0))=\max_{t\in[ a_{k},{b_{k}]}} \delta_{\Omega} (\gamma_{k}(t)) \rightarrow 0   , \mbox{ as } k \rightarrow \infty $$
	Next we will prove the following claim which will give the contradiction\\
	
	\noindent\textbf{Claim:} 
	$\delta_{\Omega}(\gamma_{k}(0))\not\rightarrow 0 \mbox{ as } k\to \infty$\\
	\noindent\textbf{Proof:}
	Since $p,q \in \partial_{Euc} \Omega, p\neq q$, therefore there exists $R>0$ and $k_0\in \mathbb{N}$ such that the following three holds
	\begin{itemize}
		\item[(i)] $|p-q|>R$,
		\item[(ii)] $|x_k-p|\le R/3$, for all $k\ge k_0$ and,
		\item[(iii)] $|y_k-p|\ge R$, for  for all $k\ge k_0$.
	\end{itemize}
	Let $\theta_{k}$ be the first point on $\gamma_{k}$ with $|\theta_{k}-p|=R$. Consider $\sigma_k=\gamma_{k}[x_k,\theta_{k}]$ and let $z_k\in \sigma_k$ be a point with $|z_k-p|=2R/3$. Since $\sigma_k$ is also quasihyperbolic geodesic in $\Omega$ joining $x_k$ and $\theta_{k}$ and $\Omega$ is a $C$- John domain in which every quasihyperbolic geodesic is a double $B=B(C)$-cone arc, therefore we have 
	\begin{eqnarray}\label{3.2}
		B\delta_{\Omega}(z_k)&\ge& \min\{l(\gamma_{k}[x_k,z_k]),l(\gamma_{k}[z_k,\theta_{k}])\}\\ \nonumber
		&\ge&\min\{|z_k-x_k|,|z_k-\theta_{k}|\}
	\end{eqnarray}
	Now,
	\begin{eqnarray*}
		|z_k-\theta_{k}|&=&|z_k-p+p-\theta_{k}|\\ \nonumber
		&\ge& |\theta_{k}-p|-|z_k-p|=R-\frac{2R}{3}=\frac{R}{3}\\ \nonumber
		&\ge&\frac{R}{3} 
	\end{eqnarray*}
	and 
	\begin{eqnarray*}
		|x_k-z_k|&=&|x_k-p+p-z_k|\\ \nonumber
		&\ge& |x_k-p|-|z_k-p|=\frac{2R}{3}-\frac{R}{3}=\frac{R}{3} \\ \nonumber
		&\ge&\frac{R}{3} 
	\end{eqnarray*}
	Therefore, in view of \eqref{3.2} we get
	\begin{equation}\label{J-A}
		\delta_{\Omega}(z_k)\ge \frac{R}{3B}
	\end{equation}
	Since $\delta_{\Omega}(\gamma_{k}(0))\ge \delta_{\Omega}(z_k)$, therefore by \eqref{J-A} it follows that
	$$\delta_{\Omega}(\gamma_{k}(0))\not\rightarrow0\,\,\, \mbox{ as } k\to \infty.$$
	Hence we have our claim. This completes the proof.
\end{pf}

\subsection{General visibility criteria} 
	\addtocontents{toc}{\protect\setcounter{tocdepth}{1}}
	This subsection is devoted to the proof of Theorem \ref{visibility criteria}. To do so we need the following lemma and it easily follows from \eqref{int}. 
	\begin{lem}\label{****}
		Let $\phi$ as in Theorem \ref{visibility criteria}. Then for every $\epsilon>0$ there exists constants $-\infty<a^*<b^*<+\infty$ such that
		$$\int_{-\infty}^{a^*}\frac{1}{\phi^{-1}\left(\frac{|t|}{2}\right)}dt<\epsilon,$$
		$$\int_{b^*}^{+\infty}\frac{1}{\phi^{-1}\left(\frac{t}{2}\right)}dt<\epsilon.$$
	\end{lem}

The following remark will be crucial for the proof
\begin{rem}\label{qh parametrization}
	It is often convenient to parametrize an arc by quasihyperbolic length. We say that $g:[0,r]\to \Omega$ is a quasihyperbolic parametrization if $l_k(g[0,t])=t$ for all $t\in [0,r]$. Then $r=l_k(g)$ and
	$$|g'(t)|=\delta_{\Omega}(g(t))$$ almost everywhere. Every rectifiable arc $\gamma\subset\Omega$ has a quasihyperbolic parametrization $g:[0,r]\to \gamma$, and $g$ satisfies the  Lipschitz condition 
	$$|g(s)-g(t)|\le M|s-t|$$
	where $M=\max\{\delta(x):x\in \gamma\}$. 
\par If $\gamma$ is a quasihyperbolic geodesic, then $g:[0,r]\to (\Omega,k_{\Omega})$ is an isometry and we say that $g$ is a quasihyperbolic geodesic path from $g(0)$ to $g(r)$. Then in view of Martin\cite[4.8]{Martin-1985}, $g$ is $C^1$.
\end{rem}

\begin{pf}[\textbf{Proof of Theorem \ref{visibility criteria}}]
Suppose $\Omega$ is not a QH-visibility domain i.e. there does not exist any compact set with the desired property. Then there are $p,q\in \partial_{Euc} \Omega$, $p\neq q$, sequences $\{x_{n}\}, \{y_{n}\} \subset \Omega$ with $x_{n} \rightarrow p$ and $y_{n} \rightarrow q$ as $n\to \infty$, and a sequence of quasihyperbolic geodesics $\gamma_{n}:[a_n,b_n]\to \Omega$ joining $x_k$ and $y_k$ such that
\begin{equation}\label{**}
	\max_{t\in[a_{n},b_{n}]}\delta_{\Omega} (\gamma_{n}(t)) \rightarrow 0   , \mbox{ as } n \rightarrow \infty
\end{equation}
By a suitable {\it quasihyperbolic parameterization} $g_n:[a_{n},b_{n}] \rightarrow \Omega$ of $\gamma_{n}$ with $g_{n}(a_{n}) = x_{n}$, $g_{n}(b_{n}) = y_{n}$ and $|b_n-a_n|=l_{k}(g_n[a_n,b_n])$, we can assume that, for all $n, a_n\le 0\le b_n$ and
	\begin{equation}\label{***}
		\max_{t\in[a_{n},b_{n}]} \delta_{\Omega} (g_{n}(t))=\delta_{\Omega}(g_{n}(0))
	\end{equation}
Furthermore, by Remark \ref{qh parametrization} we have that each $g_n$ satisfies the following Lipschitz condition 
$$|g_n(s)-g_n(t)|\le M_n|s-t|,\,\,\, s,t\in [a_n,b_n],$$
where, $M_n=\max\{\delta_{\Omega}(x):x\in \gamma_n\}$. In view of \eqref{**}, we have that $M_n\to 0$. This gives us that each $g_n$ is $M$-Lipschitz with respect to the Euclidean distance, where $M=\sup_{n\ge 1}\{M_n\}$ {\it i.e.,} for all $n$, we have
$$|g_n(s)-g_n(t)|\le M|s-t|,\,\,\, s,t\in [a_n,b_n].$$
 Therefore, by Arzel\'a-Ascoli's theorem, up to subsequence we can assume:
	\begin{itemize}
		\item[(i)] $a_{k}\rightarrow a \in [-\infty,0]$ and $b_{k} \rightarrow b \in [0,\infty]$
		\item[(ii)] $g_{k}$ converges uniformly on every compact subset of $(a,b)$ to a continuous map  $g:(a,b)\rightarrow \overline{\Omega}^{Euc}$.
		\item[(iii)] $g_{k}(a_{k}) \rightarrow p$ and $g_{k}(b_{k}) \rightarrow q$
	\end{itemize}
We made the following observation
\begin{observation}\label{O}
We conclude from the fact that
$$|g_n(a_n)-g_n(b_n)|\le M|a_n-b_n|,\,\,\, \mbox{ for all }n$$
that $a<b$.	By Remark \ref{qh parametrization}, we have that $g_n:[a_n,b_n]\to (\Omega,k_{\Omega})$ is an isometry. This gives us that 
	$$|b-a|=\lim_{k\to \infty}|b_k-a_k|=\lim_{k\to \infty}k_{\Omega}(g_{k}(b_k),g_{k}(a_k))=\infty\,\,\, (\mbox{ by Observation \ref{observation}}).$$
	Therefore, it is clear that both $a$ and $b$ together can't be finite.
\end{observation}
\begin{claim}
$g:(a,b) \rightarrow \overline{\Omega}^{Euc}$ is a constant map.
\end{claim}
\begin{pf}[Proof of Claim 1]
	In view of \eqref{**}, we have that $\delta_{\Omega}(g_k(t))$ converges uniformly to $0$. Then for $s < t \in (a,b)$, we have
	$$|g(s) - g(t)|=\lim_{k \rightarrow \infty} |g_{k}(s) - g_{k}(t)|\le\lim\limits_{n\to \infty}M_k|s-t|=0.$$
This yields that $g$ is a constant map.
\end{pf}

Further using the property of the domain $\Omega$, we will establish a contradiction by proving the following
	\begin{claim}
		$g :(a,b) \rightarrow \overline{\Omega}^{Euc}$ is not a constant map.
	\end{claim}
\begin{pf}[Proof of Claim 2]
Since $\Omega$ satisfies
\begin{equation*}
	k_{\Omega}(x_{0},x) \leq \phi\left(\frac{\delta_{\Omega}(x_0)}{\delta_{\Omega}(x)}\right) \mbox{ for all } x \in \Omega,
\end{equation*}
and $g_k$ is an isometry, we have
\begin{eqnarray*}
	|t|=k_\Omega (g_{k}(0),g_{k}(t)) &\leq& k_\Omega (g_{k}(0),x_{0})) + k_\Omega (x_{0},g_{k}(t))\\ \nonumber
	&\leq& \phi\left(\frac{\delta_{\Omega}(x_0)}{\delta_{\Omega}(g_k(0))}\right)+\phi\left(\frac{\delta_{\Omega}(x_0)}{\delta_{\Omega}(g_k(t))}\right)\\ \nonumber
	&\leq& 2\phi\left(\frac{\delta_{\Omega}(x_0)}{\delta_{\Omega}(g_k(t))}\right)\,\,\, (\mbox{by }\eqref{***})
\end{eqnarray*}
In view of Observation \ref{O}, we have only following two cases to consider:
\begin{itemize}
	\item[Case 1:] Both $a$ and $b$ are infinite.
	\item[Case 2:] Either only $a$ or only $b$ is finite. In other words, Either only $a$ or only $b$ is infinite.
\end{itemize}	
Let us first consider the Case 1 and assume that $a=-\infty$ and $b=\infty$. The fact that $b_n\to \infty$ and $a_n\to -\infty$ and the function $\phi$ is strictly increasing with $\lim\limits_{t\to\infty}\phi(t)=\infty$ imply that there is a positive integer $n_0$ and positive constants $A$ and $B$ such that we have the following
\begin{itemize}
	\item[(i)] whenever $t\in (B,b_n]$ and $n\ge n_0$, we have 
	$$\frac{|t|}{2}\in \mbox{\textbf{range}}(\phi), \mbox{ and hence, } \delta_{\Omega}(g_k(t))\le \frac{\delta_{\Omega}(x_0)}{\phi^{-1}\left(\frac{|t|}{2}\right)}$$
	\item[(ii)] whenever $t\in (a_n, -A]$ and $n\ge n_0$, we have 
$$\frac{|t|}{2}\in \mbox{\textbf{range}}(\phi), \mbox{ and hence, } \delta_{\Omega}(g_k(t))\le \frac{\delta_{\Omega}(x_0)}{\phi^{-1}\left(\frac{|t|}{2}\right)}$$
\end{itemize}In view of Lemma \ref{****}, let $\epsilon=||p-q||/2$, then we can choose $a^*\in (-\infty,-A)$ and $b^*\in (B,\infty)$ such that
\begin{equation}\label{2}
	||p-q|| >\int_{b^*}^{+\infty} \frac{\delta_{\Omega}(x_0)}{\phi^{-1}\left(\frac{t}{2}\right)}\,dt+\int_{-\infty}^{a^*} \frac{\delta_{\Omega}(x_0)}{\phi^{-1}\left(\frac{|t|}{2}\right)}\,dt
\end{equation}
Therefore, we have
\begin{eqnarray*}
	||g(b^*) - g(a^*)||&=&\lim_{k\to \infty}||g_{k}(b^*) - g_k(a^*)||\\ \nonumber
	&\geq& \limsup_{k \rightarrow \infty} (||g_{k}(b_{k}) - g_{k}(a_{k})|| - || g_{k}(b_{k}) - g_{k}(b^*)|| - || g_{k}(a_{k})-\gamma_{k}(a^*)||)\\ \nonumber
	&\geq& ||p-q|| - \limsup_{k\rightarrow \infty}\int_{b^*}^{b_{k}}|| g_{k}'(t)||\,dt - \limsup_{k \rightarrow \infty}\int_{a_k}^{a^*} ||g_{k}'(t)||\,dt\\ \nonumber
	&\ge&||p-q||-\limsup_{k\to \infty}\int_{b^*}^{b_k}\delta_{\Omega}(g_k(t))\,dt-\limsup_{k\to \infty}\int_{a_k}^{a^*}\delta_{\Omega}(g_k(t))\,dt\\ \nonumber
	&\ge&||p-q||-\limsup_{k \rightarrow \infty}\int_{b^*}^{b_k} \frac{\delta_{\Omega}(x_0)}{\phi^{-1}\left(\frac{t}{2}\right)}\,dt-\limsup_{k \rightarrow \infty}\int_{a_k}^{a^*} \frac{\delta_{\Omega}(x_0)}{\phi^{-1}\left(\frac{|t|}{2}\right)}\,dt\\ \nonumber
	&\ge&	||p-q|| -\int_{b^*}^{+\infty} \frac{\delta_{\Omega}(x_0)}{\phi^{-1}\left(\frac{t}{2}\right)}\,dt-\int_{-\infty}^{a^*} \frac{\delta_{\Omega}(x_0)}{\phi^{-1}\left(\frac{|t|}{2}\right)}\,dt>0\,\,\, (\mbox{by }\eqref{2}).
\end{eqnarray*}
Thus, $g$ is non-constant.
Now, let us consider the Case 2 and assume that $a$ is finite, and hence then $b$ is infinite. Proceed as in the proof of the Case 1, we get that $g$ is non-constant.
\end{pf}

Combining both the cases the final conclusion is that $g$ is non-constant, which contradicts Claim 1. This gives the existence of compact set $K$. Hence $\Omega$ is a QH-visibility domain. This completes the proof of Theorem \ref{visibility criteria}.
	\end{pf}

\begin{rem}\label{List}
	Our general visibility criteria gives rise to a very rich collection of QH-visibility domains, which we will mention below.
	\begin{itemize}
		\item[(1)] From general visibility criteria it follows that every $\beta$-QHBC is a QH-visibility domain by taking $\phi(t)=1/\beta\log t+C_0$.
		\item[(2)] In view of Remark \ref{John-QHBC}, by taking $\phi(t)=C\log t +A$ in the general visibility criteria, it follows that every bounded $C$-John domain is a QH-visibility domain.
	\end{itemize}
	
\end{rem}

	\addtocontents{toc}{\protect\setcounter{tocdepth}{1}}

\section{Visibility of hyperbolic and quasihyperbolic metric}
In this section, we provide the proof of Theorem \ref{GH-H-QH} and Theorem \ref{growth-H-visibility}. To prove Theorem \ref{GH-H-QH}, we first need to recall the following very interesting result of Buckley and Herron (see \cite[Theorem A]{Buckley-2020}).

\begin{Thm}\cite[Theorem A]{Buckley-2020}\label{BH}
	For each $A\ge 1$ there are explicit constants $H$ and $K$ that depends only on $A$ such that for any hyperbolic $A$-quasigeodesic $\gamma^{h}$ and any quasihyperbolic $A$-quasigeodesic $\gamma^k$ both with ends point $a$ and $b$,
	$$l_k(\gamma^h)\le K\,k(a,b)\,\,\, \mbox{ and }\,\,\,l_h(\gamma^k)\le H\,h(a,b).$$
\end{Thm} 
Using Theorem D, it is easy to establish Theorem \ref{GH-H-QH}.

	\begin{pf}[\textbf{Proof of Theorem \ref{GH-H-QH}}]
		Let $\Omega$ be a bounded hyperbolic domain which is Gromov hyperbolic. Suppose $\Omega$ is H-visibility domain. Let $p,q\in \partial_{Euc}\Omega$, $p\ne q$ and $x_k,y_k$ be any sequence in $\Omega$ such that $x_k\to p$ and $y_k\to q$. Let $\gamma_{k}$ be a sequence of quasihyperbolic geodesics joining $x_k$ and $y_k$. Then by Theorem D, we have 
		$$l_h(\gamma_{k})\le H\,h(x_k,y_k).$$
		That is $\gamma_{k}$ is hyperbolic $H$-quasigeodesic. Let $\sigma_k$ be any sequence of hyperbolic geodesics joining $x_k$ and $y_k$. Since $\Omega$ is a  H-visibility domain, for sufficiently large $k$, $\sigma_k$ intersects with a common compact set $K$. Since $\Omega$ is Gromov hyperbolic, there exists some $R$ depending on $H$ and Gromov hyperbolic constant such that the Hausdoreff distance between $\sigma_i$ for any fixed $i$ and $\gamma_{k}$ for all $k$, is less than $R$. Therefore, $\gamma_{k}$ also intersects with the same compact set $K$ for sufficiently large $k$, and hence $\Omega$ is QH-visibility domain. Conversely, suppose $\Omega$ is QH-visibility domain, following the same argument as above we conclude that $\Omega$ is H-visibility domain.
	\end{pf}
	
\begin{pf}[\textbf{Proof of Theorem \ref{growth-H-visibility}}]	
To prove Theorem \ref{growth-H-visibility}, we will consider quasihyperbolic parametrization of hyperbolic geodesics. Suppose $\Omega$ is not a H-visibility domain i.e. there exist $p,q\in \partial_{Euc} \Omega$, $p\neq q$, sequences $\{x_{n}\}, \{y_{n}\} \subset \Omega$ with $x_{n} \rightarrow p$ and $y_{n} \rightarrow q$ as $n\to \infty$, and a sequence of hyperbolic geodesics $\gamma_{n}:[a_n,b_n]\to \Omega$ joining $x_k$ and $y_k$ such that
\begin{equation}\label{**}
	\max_{t\in[a_{n},b_{n}]}\delta_{\Omega} (\gamma_{n}(t)) \rightarrow 0   , \mbox{ as } n \rightarrow \infty
\end{equation}

By a suitable {\it quasihyperbolic parameterization} $g_n:[a_{n},b_{n}] \rightarrow \Omega$ of $\gamma_{n}$ with $g_{n}(a_{n}) = x_{n}$, $g_{n}(b_{n}) = y_{n}$ and $|b_n-a_n|=l_{k}(g_n[a_n,b_n])$, we can assume that, for all $n, a_n\le 0\le b_n$ and
\begin{equation}\label{***}
	\max_{t\in[a_{n},b_{n}]} \delta_{\Omega} (g_{n}(t))=\delta_{\Omega}(g_{n}(0))
\end{equation}
Furthermore, by Remark \ref{qh parametrization} we have that each $g_n$ satisfies the following Lipschitz condition 
$$|g_n(s)-g_n(t)|\le M_n|s-t|,\,\,\, s,t\in [a_n,b_n],$$
where, $M_n=\max\{\delta_{\Omega}(x):x\in \gamma_n\}$. In view of \eqref{**}, we have that $M_n\to 0$. This gives us that each $g_n$ is $M$-Lipschitz with respect to the Euclidean distance, where $M=\sup_{n\ge 1}\{M_n\}$ {\it i.e.,} for all $n$, we have
$$|g_n(s)-g_n(t)|\le M|s-t|,\,\,\, s,t\in [a_n,b_n].$$
Therefore, by Arzela-Ascoli theorem, up to subsequence we can assume:
\begin{itemize}
	\item[(i)] $a_{k}\rightarrow a \in [-\infty,0]$ and $b_{k} \rightarrow b \in [0,\infty]$
	\item[(ii)] $g_{k}$ converges uniformly on every compact subset of $(a,b)$ to a continuous map  $g:(a,b)\rightarrow \overline{\Omega}^{Euc}$.
	\item[(iii)] $g_{k}(a_{k}) \rightarrow p$ and $g_{k}(b_{k}) \rightarrow q$
\end{itemize}
We made the following observation
\begin{observation}\label{O}
	We conclude from the fact that
	$$|g_n(a_n)-g_n(b_n)|\le M|a_n-b_n|,\,\,\, \mbox{ for all }n$$
	that $a<b$.	By Remark \ref{qh parametrization} and Theorem D, we have that $g_n:[a_n,b_n]\to (\Omega,k_{\Omega})$ is a $(K,0)$-quasi-isometry. This gives us that 
	$$|b-a|=\lim_{k\to \infty}|b_k-a_k|\ge \frac{1}{K}\lim_{k\to \infty}k_{\Omega}(g_{k}(b_k),g_{k}(a_k))=\infty\,\,\, (\mbox{ by Observation \ref{observation}}).$$
	Therefore, it is clear that both $a$ and $b$ together can't be finite.
\end{observation}
\begin{claim}
	$g:(a,b) \rightarrow \overline{\Omega}^{Euc}$ is a constant map.
\end{claim}
\begin{pf}[Proof of Claim 1]
	In view of \eqref{**}, we have that $\delta_{\Omega}(g_k(t))$ converges uniformly to $0$. Then for $s < t \in (a,b)$, we have
	\begin{eqnarray*}
		||g(s) - g(t)|| &=& \lim_{k \rightarrow \infty} || g_{k}(s) - g_{k}(t)||\\ \nonumber
		&\leq& \lim_{k \rightarrow \infty} \int_{s}^{t} || g_{k}'(t)||dt = \lim_{k \rightarrow \infty}\int_{s}^{t}\delta_{\Omega}(g_k(t))=0
	\end{eqnarray*}
	This yields that $g$ is a constant map.
\end{pf}

Further using the property of the domain $\Omega$, we will establish a contradiction by proving the following
\begin{claim}
	$g :(a,b) \rightarrow \overline{\Omega}^{Euc}$ is not a constant map.
\end{claim}
\begin{pf}[Proof of Claim 2]
	Since $\Omega$ satisfies
	\begin{equation*}
		k_{\Omega}(x_{0},x) \leq \phi\left(\frac{\delta_{\Omega}(x_0)}{\delta_{\Omega}(x)}\right) \mbox{ for all } x \in \Omega,
	\end{equation*}
	and each $g_k$ is a $(K,0)$-quasi-isometry, therefore we have
	\begin{eqnarray*}
		\frac{|t|}{K}\le k_\Omega (g_{k}(0),g_{k}(t)) &\leq& k_\Omega (g_{k}(0),x_{0})) + k_\Omega (x_{0},g_{k}(t))\\ \nonumber
		&\leq& \phi\left(\frac{\delta_{\Omega}(x_0)}{\delta_{\Omega}(g_k(0))}\right)+\phi\left(\frac{\delta_{\Omega}(x_0)}{\delta_{\Omega}(g_k(t))}\right)\\ \nonumber
		&\leq& 2\phi\left(\frac{\delta_{\Omega}(x_0)}{\delta_{\Omega}(g_k(t))}\right)\,\,\, (\mbox{by }\eqref{***})
	\end{eqnarray*}
	Thus,
	$$\delta_{\Omega}(g_{k}(t))\le \frac{\delta_{\Omega}(x_0)}{\phi^{-1}\left(\frac{|t|}{2K}\right)}$$
	From here following the same line as in the proof of Theorem \ref{visibility criteria}, we conclude that $g$ is non-constant.
\end{pf}

This contradicts Claim 1, and hence $\Omega$ is a H-visibility domain. This completes the proof of Theorem \ref{growth-H-visibility}.
\end{pf}

	\section{Visibility of domains quasiconformally equivalent to the unit ball.}\label{QC to visibility}
	In this section, as an application of Theorem \ref{Main-thm-GB-EB} and Theorem \ref{no loop}, we prove Theorem \ref{LC along boundry} and Corollary \ref{LC boundary}. To proceed further we recall some results related to the boundary extension of conformal and quasiconformal mappings which requires the following definitions. 
	\begin{defn}
		A bounded domain $\Omega\in \mathbb{R}^n$, $n\ge 2$, is said to be
		\begin{itemize}
			\item[(i)] finitely connected along the boundary if every point $p\in \partial_{Euc}\Omega$ has sufficiently small neighborhoods $U$ such that $\Omega\cap U$ has finite number of components,
			\item[(ii)] locally connected along the boundary if every point $p\in \partial_{Euc}\Omega$ has sufficiently small neighborhoods $U$ such that $\Omega\cap U$ is connected,
			\item[(iii)] Jordan domain if $\partial_{Euc}\Omega$ is homeomorphic to $S^{n-1}$.
		\end{itemize}
	\end{defn}
%\begin{rem}
%	Every Jordan domain in $\mathbb{R}^2$ is simply connected domain. Moreover, a proper simply connected domain in $\mathbb{R}^2$ is Jordan domain if, and only if, $\Omega$ is locally conncted along its boundary.
%\end{rem}
	
We first recall some classical results related to the extension of conformal mappings of the unit disk. We will state these results for bounded domains only. For details we refer to  \cite[Chapter IX.4]{Palka-Book} and \cite{Vaisala-Book}.
Let us recall the following well-known classical result. 

	\begin{Thm}[Carath\'eodory-Osgood-Taylor theorem]\label{COT}
	Let $\Omega$ be a bounded domain in $\mathbb{C}$, and let $f$ be a conformal mapping from $\mathbb{D}$ to $\Omega$. Then $f$ extends to a homeomorphism $\overline{f}$ from $\overline{\mathbb{D}}^{Euc}$ to $\overline{\Omega}^{Euc}$ if, and only if, $\Omega$ is a Jordan domain.
\end{Thm}
	In 1913, Osgood and Taylor proved Theorem E. Carath\'eodory proved Theorem E independently by introducing his prime end theory. We now mention a result which is due to V\"ais\"al\"a and N\"akki (see \cite[Theorem 4.7, Chapter IX.4]{Palka-Book}).
	
	\begin{Thm}
		Let $\Omega$ be a bounded domain in $\mathbb{C}$, and let $f$ be a conformal mapping from $\mathbb{D}\to \Omega$. Then $f$ extends to a continuous mapping $\overline{f}$ of $\overline{\mathbb{D}}^{Euc}$ onto $\overline{\Omega}^{Euc}$ if, and only if, $\Omega$ is finitely connected along the boundary.
	\end{Thm}
	
	Carath\'eodory has proved the following variant of Theorem B (see \cite[Chapter IX.4]{Palka-Book}).
	
	\begin{Thm}
		Let $\Omega$ be a bounded domain in $\mathbb{C}$, and let $f$ be a conformal mapping from $\mathbb{D}\to \Omega$. Then $f$ extends to a continuous mapping $\overline{f}$ of $\overline{\mathbb{D}}^{Euc}$ onto $\overline{\Omega}^{Euc}$ if, and only if, $\partial_{Euc}\Omega$ is locally connected.
	\end{Thm}
	
\begin{rem}\label{Rem 1}
	Thus, in view of Theorem C and Theorem D we have that in the case of planar simply connected domains, finitely connected along the boundary is equivalent to boundary being locally connected.
\end{rem} 

Further, if we merely replace the requirement that the extension is homeomorphism, then we have the following.
	
	\begin{Thm}
		Let $\Omega$ be a bounded domain in $\mathbb{C}$, and let $f$ be a conformal mapping from $\mathbb{D}\to \Omega$. Then $f$ extends to a homeomorphism $\overline{f}$ from $\overline{\mathbb{D}}^{Euc}$ to $\overline{\Omega}^{Euc}$ if, and only if, $\Omega$ is locally connected along the boundary.
	\end{Thm}

\begin{rem}\label{Rem 2}
In view of Jordan curve theorem we know that every Jordan domain in $\mathbb{C}$ is simply connected and a proper simply connected domain is Jordan domain if, and only if, it is locally connected along the boundary. 
\end{rem}

In the case of quasiconformal mappings, we have the following result (see \cite{Vaisala-Book} or \cite[Theorem 6.5.8 and Theorem 6.5.11]{Gehring-Palka-Martin-Book}).
	
	\begin{Thm}
		Let $\Omega$ be a bounded domain such that $\Omega$ is quasiconformally equivalent to the unit ball $\mathbb{B}^n$. Then every quasiconformal mapping $f:\mathbb{B}^n\to \Omega$ extends to a continuous map from $\overline{\mathbb{B}^n}^{Euc}$ onto $\overline{\Omega}^{Euc}$ if, and, only if, $\Omega$ is finitely connected along the boundary. Moreover, $f$ extends to a homeomorphism if, and, only if, $\Omega$ is locally connected along the boundary.
	\end{Thm}
	
We are now ready to prove Theorem \ref{LC along boundry} and Corollary \ref{LC boundary}.
	
	\begin{pf}[\textbf{Proof of Theorem \ref{LC along boundry}}]
		Let $\Omega$ be a bounded domain which is quasiconformally equivalent to the unit ball {\it i.e.},	there exists a quasiconformal mapping $f:\mathbb{B}^n\to \Omega$. Since quasiconformal maps are rough quasi-isometries in the quasihyperbolic metric, each quasiconformal image of Gromov hyperbolic domain is Gromov hyperbolic. Therefore, $(\Omega,k_{\Omega})$ is Gromov hyperbolic.
		\par \textbf{Proof of $(i)\implies (ii)$:} Suppose $\Omega$ is a QH-visibility domain. In view of Theorem \ref{Main-thm-GB-EB}, we have that the identity map 
		$$id_{\Omega}:(\Omega,k_{\Omega}) \to (\Omega,d_{Euc})$$
		extends to a continuous surjective map 
		$$\widehat{id_{\Omega}}:\overline{\Omega}^G\to \overline{\Omega}^{Euc}.$$ 
		Also, since $\mathbb{B}^n$ is a QH-visibility domain with no geodesic loop, the identity map
		$$id_{\mathbb{B}^n}:(\mathbb{B}^n,k_{\mathbb{B}^n}) \to (\mathbb{B}^n,d_{Euc}).$$
		extends to a homeomorphism
		$$\widehat{id_{\mathbb{B}^n}}:\overline{\mathbb{B}^n}^G\to \overline{\mathbb{B}^n}^{Euc}.$$ 
		Further, since $(\mathbb{B}^n,k_{\mathbb{B}^n})$ and $(\Omega,k_{\Omega})$ are Gromov hyperbolic and $f$ is rough quasiisometry with respect to quasihyperbolic metric, we have that $f$ extends to a homeomorphism
		$\widetilde{f}:\overline{\mathbb{B}^n}^G\to \overline{\Omega}^G$. 	Therefore, $f:\mathbb{B}^n\to \Omega$ extends to a continuous surjective map 
		$$\widehat{f}=\widehat{id_{\Omega}}\circ\widetilde{f}\circ\widehat{id_{\mathbb{D}}}^{-1}:\overline{\mathbb{B}^n}^{Euc}\to \overline{\Omega}^{Euc}.$$ 
		Hence by Theorem F, $\Omega$ is finitely connected along the boundary.
		\par  \textbf{Proof of $(ii)\implies(i)$:} Suppose $\Omega$ is finitely connected along the boundary, then by Theorem H, $f:\mathbb{B}^n\to \Omega$ extends to a continuous surjective map
		$$\widetilde{f}:\overline{\mathbb{B}^n}^{Euc}\to \overline{\Omega}^{Euc}.$$ 
		From the theory of Gromov hyperbolic spaces $f^{-1}$ extends to a homeomorphism 
		$$\widehat{f^{-1}}:\overline{\Omega}^G\to \overline{\mathbb{B}^n}^{G}.$$
		Also since $\mathbb{B}^n$ is a QH-visibility domain with no geodesic loop, the identity map 
		$$id_{\mathbb{B}^n}:(\mathbb{B}^n,k_{\mathbb{B}^n}) \to (\mathbb{B}^n,d_{Euc})$$
		extends to a homeomorphism
		$$\widehat{id_{\mathbb{B}^n}}:\overline{\mathbb{B}^n}^G\to \overline{\mathbb{B}^n}^{Euc}.$$
		Therefore, the identity map 
		$$id_{\Omega}:(\Omega,k_{\Omega}) \to (\Omega,d_{Euc})$$
		extends to a continuous surjective map
		$$\widehat{id_{\Omega}}=\widetilde{f}\circ\widehat{id_{\mathbb{B}^n}}\circ\widehat{f^{-1}}:\overline{\Omega}^G\to \overline{\Omega}^{Euc}.$$
		In view of Theorem \ref{Main-thm-GB-EB}, $\Omega$ is a QH-visibility domain. 
		\par Moreoever, suppose $\Omega$ is a QH-visibility domain with no geodesic loop in $\overline{\Omega}^{Euc}$, we have that $id_{\Omega}$ extends to a homeomorphism from $\overline{\Omega}^G$ onto $\overline{\Omega}^{Euc}$. Therefore following the proof of $(i)\implies (ii)$, we have that $f:\mathbb{B}^n\to \Omega$ extends to a homeomorphism 
		$$\widehat{f}=\widehat{id_{\Omega}}\circ\widetilde{f}\circ\widehat{id_{\mathbb{D}}}^{-1}:\overline{\mathbb{B}^n}^{Euc}\to \overline{\Omega}^{Euc}.$$ 
			Hence by Theorem H, $\Omega$ is locally connected along the boundary. Conversly, suppose $\Omega$ is locally connected along the boundary. By following the proof of $(ii)\implies (i)$, we have that, the identity map 
			$$id_{\Omega}:(\Omega,k_{\Omega}) \to (\Omega,d_{Euc})$$
			extends to a homeomorphism
			$$\widehat{id_{\Omega}}=\widetilde{f}\circ\widehat{id_{\mathbb{B}^n}}\circ\widehat{f^{-1}}:\overline{\Omega}^G\to \overline{\Omega}^{Euc}.$$
			Hence by Theorem \ref{no loop}, $\Omega$ is a QH-visibility domain with no geodesic loop in $\overline{\Omega}^{Euc}$. This completes the proof.
	\end{pf}

	\begin{pf}[\textbf{Proof of Corollary \ref{LC boundary}}]
	Let $\Omega$ be a bounded simply connected domain. Then by Theorem \ref{LC along boundry}, $\Omega$, is QH-visibility domain if, and only if, $\Omega$ is finitely connected along the boundary, which by Remark \ref{Rem 1}, is equivalent to $\partial \Omega$ is locally connected. Moreover, by Theorem \ref{LC along boundry} $\Omega$ is QH-visibility domain with no geodesic loop in $\overline{\Omega}^{Euc}$ if, and only if, $\Omega$ is locally connected along the boundary, which by Remark \ref{Rem 2}, is equivalent to $\Omega$ is a Jordan domain. This completes the proof.
	\end{pf}

	\section{Continuous extension of quasihyperbolic isometries and quasi-isometries}\label{Isometry and quasiisometry}
	This section is devoted to the proof of Theorem \ref{Isometry}, Theorem \ref{quasiisometry} and Theorem \ref{quasiconformal}.
	
	\begin{pf}[\textbf{Proof of Theorem \ref{Isometry}}]
			Fix some point $p\in \partial_{Euc}\Omega$. We claim that 
		$$\lim_{x\to p} f(x) \mbox{ exists in } \partial_{Euc}\Omega'.$$
		If $\lim_{x\to p} f(x)$ does not exist, then there exist sequences $x_k,y_k\in\Omega$ with 
		$$\lim_{k\to \infty}x_k=p=\lim_{k\to \infty}y_k \mbox{ but } \lim_{k\to \infty}f(x_k)=\xi\ne \eta=\lim_{k\to \infty}f(y_k).$$
		Since $f$ is quasihyperbolic isometry, $\eta,\xi\in \partial_{Euc}\Omega'$. For each $k$, let $\gamma_k$ be a sequence of geodesic in $\Omega$ joining $x_k$ and $y_k$. Then $f\circ\gamma_k$ is also a quasihyperbolic geodesic in $\Omega'$. Fix $x_0\in \Omega$. Since $\Omega'$ is a QH-visibility domain, In view of Lemma \ref{iff}, we have
		$$\sup_{k\ge 1}k_{\Omega'}(f(x_0),f\circ\gamma_k)<\infty.$$
		Hence
		$$\sup_{k\ge 1}k_{\Omega}(x_0,\gamma_k)<\infty,$$
		which is a contradiction to the fact that $\Omega$ is quasihyperbolically well behaved. Thus, $\lim_{x\to p} f(x) \mbox{ exists in } \partial_{Euc}\Omega'.$
		\par Next we define the map $F:\overline{\Omega}^{Euc}\to \overline{\Omega'}^{Euc}$ by
		\[F(z)=\begin{cases}
			f(z)&\mbox{ if } z\in \Omega \\[5mm]
			\widetilde{f}(z)=\lim_{x\to z}f(z)&\mbox{ if } z\in\partial_{Euc}\Omega
		\end{cases}
		\]
		In view of the definition of $F$, to show $F$ is continuous, it is sufficient to show that $\widetilde{f}:\partial_{Euc}\Omega\to \partial_{Euc}\Omega'$ is continuous. Let $p_k$ be a sequence in $\partial_{Euc}\Omega$ which converges to $p\in \partial_{Euc}\Omega$, since $p\in \partial_{Euc}\Omega$, we can find a sequence $z_k\in \Omega$ which is sufficiently close to $p_k$ such that $z_k\to p$. Then write $\lim_{k\to \infty}f(z_k)=\widetilde{f}(p)=q\in \partial_{Euc}\Omega'$. is a sequence in $\partial_{Euc}\Omega'$. Since $\partial_{Euc}\Omega'$ is compact and $z_k$ is sufficiently close to $p_k$, we have
		$$\lim_{k\to\infty}\widetilde{f}(p_k)=\lim_{k\to \infty}f(z_k)=q=\widetilde{f}(p).$$
		This completes the proof.
	\end{pf}

\begin{pf}[\textbf{Proof of Theorem \ref{quasiisometry}}]
	Fix some point $p\in \partial_{Euc}\Omega$. We claim that 
$$\lim_{x\to p} f(x) \mbox{ exists in } \partial_{Euc}\Omega'.$$
If $\lim_{x\to p} f(x)$ does not exist, then there exists sequence $x_k,y_k\in\Omega$ with 
$$\lim_{k\to \infty}x_k=p=\lim_{k\to \infty}y_k \mbox{ but } \lim_{k\to \infty}f(x_k)=\xi\ne \eta=\lim_{k\to \infty}f(y_k)$$
Since $f$ is  isometry, $\eta,\xi\in \partial_{Euc}\Omega'$. For each $k$, let $\gamma_k$ be a sequence of geodesic in $\Omega$ joining $x_k$ and $y_k$. Then $f\circ\gamma_k=\sigma_k:[a_k,b_k]\to \Omega'$ is a sequence of $(1,\mu)$-quasi-isometric path in $\Omega'$ joining $f(x_k)=u_k$ and $f(y_k)=v_k$.
\subsection*{Claim 1:} There exists a compact set $K\subset \Omega$ such that for sufficiently large $k$, $\sigma_k$ intersects with $K$. 
\subsection*{Proof of the Claim 1:}
If $\sigma_k$ there does not exist such compact set, then
$$\max_{t\in[ a_{k},b_{k}]}\delta_{\Omega'}(\sigma_k(t)) \to 0\,\, \mbox{ as }\,\, n \to \infty$$ 
Let $w_k=\sigma_k(t_k)$ such that
$$\max_{t\in[ a_{k},b_{k}]}\delta_{\Omega'}(\sigma_k(t))=\delta_{\Omega'}(\sigma_k(t_k)).$$
Let $f(x_0)=w_0$. Clearly, $z_k\to \xi\in\partial_{Euc}\Omega'$, therefore, $k_{\Omega'}(w_0,w_k)\to \infty$ as $k\to \infty$. Now
\begin{eqnarray*}
	k_{\Omega'}(u_k,w_0)+k_{\Omega'}(v_k,w_0)&\ge&k_{\Omega'}(f(x_k),f(y_k))\\ \nonumber
	&=&k_{\Omega'}(\sigma_k(a_k),\sigma_k(b_k))\\ \nonumber
	&\ge& (b_k-a_k)-\mu\,(\mbox{since } \sigma_k \mbox{ is a } (1,\mu)- \mbox{quasi-isometric path}\\ \nonumber
	&=&(b_k-t_k+\mu)+(t_k-a_k+\mu)-3\mu\\ \nonumber
	&\ge& k_{\Omega'}(v_k,w_k)+k_{\Omega'}(w_k,u_k)-3\mu
\end{eqnarray*}
Therefore, we have
\begin{eqnarray*}
	k_{\Omega'}(u_k,w_0)+k_{\Omega'}(v_k,w_0)&\ge& -2(v_k|w_k)_{w_0}-2(w_k|u_k)_{w_0}+k_{\Omega'}(u_k,w_0)+k_{\Omega'}(w_0,v_k)\\
	&&+2k_{\Omega'}(w_0,w_k)-3\mu
\end{eqnarray*}
Since $\Omega'$ is a visibility domain, we have
$$k_{\Omega'}(u_k,w_0)+k_{\Omega'}(v_k,w_0)\ge-2C+k_{\Omega'}(u_k,w_0)+k_{\Omega'}(w_0,v_k)+2k_{\Omega}(w_0,w_k)-3\mu,$$
which gives that 
$$2k_{\Omega'}(w_0,w_k)\le 2C+3\mu.$$
This is a contradiction to the fact that $k_{\Omega'}(w_0,w_k)\to \infty$.

\subsection*{Claim 2:} $$\sup_{k\ge 1}k_{\Omega'}(f(x_0),\sigma_k)<\infty.$$
\subsection*{Proof of Claim 2:}
Since there exists a compact set $K\subset\Omega$ such that $\sigma_k$ intersects with $K$ for sufficiently large $k$, let $z_k$ be one of the intersection points. Then for sufficiently large $k$, we have the following
$$k_{\Omega'}(w_0,z_k)\le k_{\Omega'}(w_0,z),\,\,\,\,\, \mbox{ for all } z\in \sigma_k.$$
That is, $k_{\Omega'}(w_0,\sigma_k)=k_{\Omega'}(w_0,z_k)$ for suffieiently large $k$. Since $z_k$ lies in a compact set, we have 
$$\sup_{k\ge 1}{k_{\Omega'}}(w_0,\sigma_k)\le M,$$
and hence
$$\sup_{k\ge 1}k_{\Omega}(x_0,\gamma_k)<\infty.$$	

This completes the proof of Claim 2. Now we note that Claim 2 gives a contradiction to the fact that $\Omega$ is quasihyperbolically well behaved. Thus, $\lim_{x\to p} f(x)$ exists in $\partial_{Euc}\Omega'.$
\par Next define the map $F:\overline{\Omega}^{Euc}\to \overline{\Omega'}^{Euc}$ by
\[F(z)=\begin{cases}
	f(z)&\mbox{ if } z\in \Omega \\[5mm]
	\widetilde{f}(z)=\lim_{x\to z}f(z)&\mbox{ if } z\in\partial_{Euc}\Omega.
\end{cases}
\]
In view of the definition of $F$, to show $F$ is continuous, it is sufficient to show that $\widetilde{f}:\partial_{Euc}\Omega\to \partial_{Euc}\Omega'$ is continuous. Let $p_k$ be a sequence in $\partial_{Euc}\Omega$ which converges to $p\in \partial_{Euc}\Omega$. Since $p\in \partial_{Euc}\Omega$, we can find a sequence $z_k\in \Omega$ which is sufficiently close to $p_k$ such that $z_k\to p$. Then write $\lim_{k\to \infty}f(z_k)=\widetilde{f}(p)=q\in \partial_{Euc}\Omega'$. is a sequence in $\partial_{Euc}\Omega'$. Since $\partial_{Euc}\Omega'$ is compact and $z_k$ is sufficiently close to $p_k$, we have
$$\lim_{k\to\infty}\widetilde{f}(p_k)=\lim_{k\to \infty}f(z_k)=q=\widetilde{f}(p).$$
This completes the proof of the theorem.
\end{pf}

	\begin{pf}[\textbf{Proof of Theorem \ref{quasiconformal}}]
	Since $\Omega$ and $D$ are Gromov hyperbolic, therefore in view of Theorem \ref{QI extension}, we have that $f$ extends to a homeomorphism $\widetilde{f}:\overline{\Omega}^G\to \overline{D}^G$. Since $\Omega$ and $D$ are Gromov hyperbolic and QH-visibility domains with no geodesic loops, therefore we have that the identity map
	$$id:(\Omega,k_D)\to(\Omega,d_{Euc})$$ 
	extends as a homeomorphism $$\Phi_{\Omega}:\overline{\Omega}^G\to \overline{\Omega}^{Euc}$$
	and the identity map
	$$id_D:(D,k_D)\to(D,d_{Euc})$$ 
	extends as a homeomorphism 
	$$\Phi_{D}:\overline{D}^G\to \overline{D}^{Euc}.$$
	The map
	$$F=\Phi_{D}\circ\widetilde{f}\circ\Phi_{\Omega}^{-1}:\overline{\Omega}^{Euc}\to \overline{D}^{Euc}$$
	gives the desired homeomorphic extension. This completes the proof.
	\end{pf}
	
		\section{Unbounded QH-visibility domains and Gromov hyperbolicity}\label{7}
	In this section we prove Theorem \ref{unbounded QH-Gromov}. We will divide the proof of Theorem \ref{unbounded QH-Gromov} into several Lemmas. 
		\begin{lem}\label{6.2}
		Let $\Omega\subsetneq\mathbb{R}^n$ be a QH-visibility domain. If $\gamma$ is a qh-geodesic ray, then $\lim_{t\to \infty}\gamma(t)$ exists in $\partial^{\infty}\Omega$.
	\end{lem}
	\begin{proof}
		Since $\partial^{\infty}\Omega$ is compact in $\overline{\Omega}^{\infty}$, for any $t_k\to \infty$ the sequence $\{\gamma(t_k)\}$ has a convergent subsequence. Moreover, as $\gamma$ is a qh-geodesic ray for any sequence $t_k\to \infty$, we have $|t_k|=k_{\Omega}(\gamma(0),\gamma(t_k)) \to \infty$ as $k\to \infty$. Then by Result 1, it is clear that all the limit points of $\{\gamma(t_k)\}$ belongs to $\partial^{\infty}\Omega$. In view of this consider following set 
		$$L=\{x\in \partial^{\infty}{\Omega}:\mbox{ there exists } t_k\to \infty \mbox{ such that } \gamma(t_k)\to x\}.$$ 
		Suppose for contradiction that $L$ is not sigleton in this case $L$ contains at least one point in $\partial_{Euc}\Omega$. Then, by connectedness, $L$ must contain at least two
		points in $\partial_{Euc}\Omega$, say, $p,q$. Then following the proof of Lemma \ref{lands}, we get that $L\cap K\ne \phi$, where $K$ is a compact set coming from the visibility property. This completes the proof.
	\end{proof}

	Next we show that in visibility domain, equivalent qh-geodesic has same limits. This will gives us that the map is well defined.	
	
	\begin{lem}\label{6.3}
		Let $\Omega$ be QH-visibility domain in $\mathbb{R}^n$. Fix $x_0\in \Omega$ and let $\gamma$ and $\sigma$ be any two qh-geodesic rays with $\gamma(0)=x_0=\sigma(0)$ such that $\gamma \sim \sigma$. Then 
		$$\lim_{t\to \infty}\gamma(t)=\lim_{t\to \infty} \sigma(t).$$
	\end{lem}
	
	\begin{pf}
		Let $\Omega$ be a QH-visibility domain. In view of Lemma \ref{6.2}, we have that $\lim_{t\to \infty}\gamma(t)\in \partial^{\infty}\Omega$. By contradiction suppose 
		$$\lim_{t\to \infty}\gamma(t)=p\neq q=\lim_{t\to \infty} \sigma(t).$$
		There exists a sequence $\{s_k\}$ with $s_k\to \infty$ such that $\lim\limits_{k\to \infty}\gamma(s_k)=p$ and $\lim\limits_{k\to \infty}\sigma(s_k)=q$. Write $z_k=\gamma(s_k)$ and $w_k=\sigma(s_k)$, then $z_k\to p$ and $w_k\to q$. As both $p$ and $q$ can't be infinity so assume atleast one of $p$ and $q$ belongs to $\partial\Omega$, then by Observation \ref{observation} and by the fact that $\gamma \sim \sigma$, we get a contradiction.
	\end{pf}

	\begin{pf}[\textbf{Proof of Theorem \ref{unbounded QH-Gromov}}] 
		\textbf{(ii) $\implies$ (i):} Suppose the identity map extends as a continuous surjective map $\widehat{id}:\overline{\Omega}^G \to \overline{\Omega}^{\infty}$. It is clear that $\widehat{id}(\partial_{G}\Omega)=\partial\Omega\cup \{\infty\}$. To prove $\Omega$ is a QH-visibility domain we will use Theorem A {\it i.e.}, we will show that $\partial\Omega$ is visibile. For this let $p,q\in \partial\Omega$, $p\ne q$. Consider the two subsets $L_p=\widehat{id}^{-1}(p)$, $L_q=\widehat{id}^{-1}(q)$ of $\partial_{G}\Omega$. Clearly both are compact and $L_p\cap L_q=\emptyset$. Since $(\Omega,k_{\Omega})$ is Gromov hyperbolic, it is easy to prove that there exist $\overline{\Omega}^G$-open subset $V_p$ and $V_q$ of $p$ and $q$ respectively and  compact set $K_{p,q}\subset \Omega$ such that $L_p\subset V_p$, $L_q\subset V_q$ and $\overline{V_p}\cap \overline{V_q}=\emptyset$ with the property that if $\gamma:[0,T]\to \Omega$ is any quasihyperbolic geodesic with 
		$\gamma(0)\in V_p\cap \Omega$ and $\gamma(T)\in V_q\cap \Omega$, then $\gamma([0,T])\cap K_{p,q}\ne \phi$.
		\par Now suppose $x_k$ and $y_k$ are sequences in $\Omega$ converging in the Euclidean topology to $p$ and $q$ respectively. Then there exists $k_0\in \mathbb{N}$ such that $x_k\in V_p\cap \Omega$ and $y_k\in V_q\cap \Omega$, for all $k\ge k_0$. Thus, for $k\ge k_0$ any qh-geodesic $\gamma_k$ joining $x_k$ and $y_k$ intersects with $K_{p,q}$, and hence $\Omega$ is a QH-visibility domain.\\
		
		\textbf{(i) $\implies$ (ii):}	Let $\Omega\subsetneq\mathbb{R}^n$, be a QH-visibility domain. In view of Lemma \ref{6.2}, define a map 
		$\Phi:\overline{\Omega}^G\to \overline{\Omega}^{\infty}$ by
		\[\Phi(z)=\begin{cases}
			z&\mbox{ if } z\in \Omega \\[3mm]
			p=\lim\limits_{t\to \infty}\gamma(t)&\mbox{ if } z=\xi=[\gamma]\in\partial_{G}\Omega
		\end{cases}
		\]
		In view of Lemma \ref{6.3}, $\Phi$ is well defined. Since $\overline{\Omega}^G$ and $\overline{\Omega}^{\infty}$ is compact and $\Phi$ is continuous with $\Phi(\Omega)=\Omega$ which is dense in $\overline{\Omega}^{\infty}$, theorefore $\Phi$ is onto. To show that $\Phi$ is continuous we need to consider the following two cases:
		\begin{pf}[Case I: If $\{z_k\}\subset \Omega$ is a sequence such that $z_k \stackrel{Gromov}\longrightarrow \xi$, where $\xi \in \partial_{G}\Omega$]
			Since $z_k \stackrel{Gromov}\longrightarrow \xi$, which means that for every qh-geodesic $\gamma_k:[0,T_k]\to \Omega$ with $\gamma_k(0)=z_0$ and $\gamma_k(T_k)=z_k$, every subsequence has a subsequence which converges uniformly on compact subset to a geodesic ray $\gamma\in \xi$. In view of Lemma \ref{6.2}, we have $\lim\limits_{t\to \infty}\gamma(t)\in \partial^{\infty}\Omega$. Our aim is to show that
			$$\lim\limits_{k\to \infty}z_k=\lim\limits_{t\to \infty}\gamma(t).$$
			If $\lim\limits_{t\to \infty}\gamma(t)=p\in \partial_{Euc}\Omega$, then $z_k\not\to\infty$ since $\gamma_{k}$ converges uniformly on compacta to $\gamma$. Then an argument using visibility property (see Lemma \ref{continuous}) shows that $p=q$.
			If $\lim\limits_{t\to \infty}\gamma(t)=\infty$, then $z_k\to \infty$ as $\gamma_{k}$ converges uniformly on compacta to $\gamma$.
				\end{pf}
		\begin{pf}[Case II: ] If $\{\xi_k\}\subset\partial_{G}\Omega$ is a sequence such that $\xi_k \stackrel{Gromov}\longrightarrow \xi$, where $\xi \in \partial_{G}\Omega$
		Let $\gamma_k$ be the representative from $\xi_k$. Then $\xi_k \stackrel{Gromov}\longrightarrow \xi$ gives that every subsequence of $\gamma_k$ has a subsequence which converges uniformly on compact subsets to a geodesic ray $\gamma\in \xi$. Then using the same argument as in Case I, we get the desired.
	\end{pf}

Thus, in view of Case I and Case II, $\Phi$ is continuous. This completes the proof.
\end{pf}

	\section{Some Examples}\label{Examples}
	The aim of this section is two fold. We first want to construct examples of a visibility domain which is neither John domain nor QHBC domain. Second we will mention some examples of domains which are visibility domain but not Gromov hyperbolic.
	\begin{example}[\textbf{Visibility domain which is neither John domain nor QHBC domain}] Fix $\alpha,\beta$ such that $1\le \beta<\alpha$. Let $S_0$ be a an open circular region with center at $x_0$ on imaginary axis and radius $r_0$, which hits the real axis at two points, where $r_0=1$. Let
		$$\Omega=S_0\,\cup\,R_{01}\,\cup\,S_{01}\,\cup\,R_{02}\,\cup\,S_{02}\,\cup\,R_{03}\,\cup\,S_{03}\,\cup\,\ldots\subset\mathbb{R}^2,$$
		where for each $m=0,1,2,\ldots,\,\, S_{m}$ is an open circle (red arcs in Fig 1) with center $x_m$ and radius $r_m$ and for each $m=1,2,\ldots,\,\, R_{0m}$ is an open rectangular region (blue lines in Fig 1) with height $r_m^{\beta}$ and width $r_m^{\alpha}$. We arrange these pieces so that the top face of $R_{0m}$ is contained in the boundary of $S_{0m}$ (see Figure 1). One can ensure that $S_{0m}, m=1,2,\ldots,$ are pairwise disjoint by forcing sequence $r_m$ to decrease to zero faster.

			\vspace{2 mm}
		\begin{figure}[!htb]\label{Fig1}
			\begin{center}
			\includegraphics{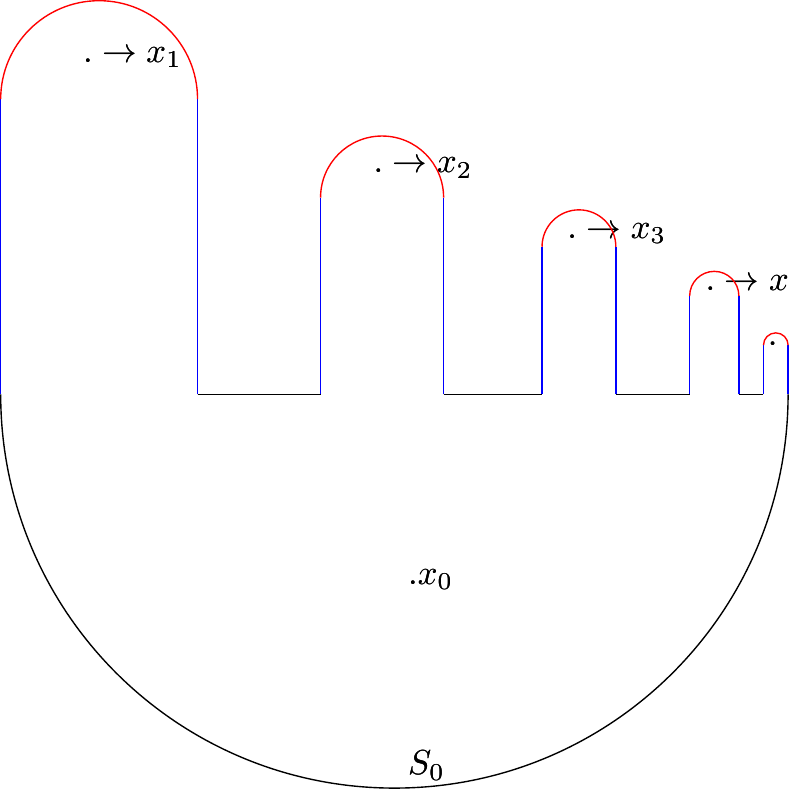}
			\end{center}
			\caption{Foot with infinite finger type domain}
		\end{figure}
		
	\end{example}

	\subsection*{Claim 1} $\Omega$ is not a John domain with center $x_0$ for any $C\ge 1$.
	\subsection*{Proof of Claim 1}
	By way of contradiction, suppose $\Omega$ is a $C$- John domain for some $C\ge 1$. Let $\gamma_m$ be a curve joining $x_0$ to $x_m$ then
	
	$$	l(\gamma_m[x_m,y])\le C\delta_{\Omega}(y)\,\,\, \mbox{ for every } y\in \gamma_m$$
	
	Let $y\in R_{0m}$ such that $y$ lies outside the circle $S_{0m}$, then 
	\begin{eqnarray}\label{fingers-1}
		l(\gamma_m[x_m,y])\le C\delta_{\Omega}(y)\le Cr_m^{\alpha}
	\end{eqnarray}
	Also in this case we have,
	\begin{equation}\label{fingers-2}
		l(\gamma_m[x_m,y])\ge \mbox{ dist}(x_m, C_{0m})=r_m.
	\end{equation}
	By \eqref{fingers-1} and \eqref{fingers-2}, we have that
	$$r_m\le Cr_m^{\alpha} \,\,\,\,\,\, \mbox{ for all } m,$$
	which is a contradiction to the fact that $\alpha>1$. Thus, $\Omega$ is not a John domain.
	
	\subsection*{Claim 2} If $\alpha-\beta>1$, then $\Omega$ is not a $B$-QHBC domain for any $B<1$.
	\subsection*{Proof of Claim 2}
	Let $\alpha-\beta>1$. By contradiction suppose $\Omega$ is $B$-QHBC domain that is
	\begin{equation}\label{EQHBC}
		k_{\Omega}(x_0,x)\le \frac{1}{B}\log\left(\frac{\delta_{\Omega}(x_0)}{\delta_{\Omega}(x)}\right)+C_0.
	\end{equation}
	Let $\gamma_m$ be any quasihyperbolic geodesic joining $x_0$ to $x_m$. Then
	\begin{equation}
		k_{\Omega}(x_0,x_m)=\int_{\gamma_m}\frac{|dz|}{\delta_{\Omega}(z)},\,\,\, z\in \gamma_m
	\end{equation}	
	Devide $\gamma_m$ into three arcs $\gamma_m^1$, $\gamma_m^2$ and $\gamma_m^3$, where $\gamma_m^1$ is the part of $\gamma_m$ which lies in $S_{0m}$ but not in $R_{0m}$; $\gamma_m^2$ is the part of $\gamma_m$ which lies only in $R_{0m}$ and $\gamma_m^3$ is the part of $\gamma_m$ which neither lies in $S_{0m}$ nor in $R_{0m}$ (see Fig 2)
	
		\begin{figure}[!htb]\label{Fig 2}
		\begin{center}
		\includegraphics{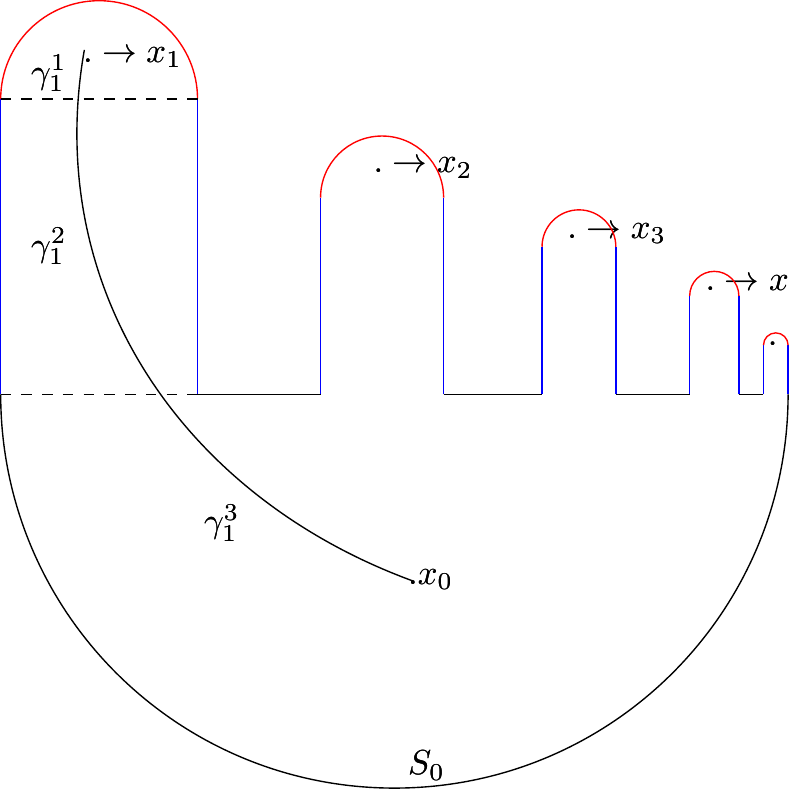}
		\end{center}
		\caption{Three parts of $\gamma_m$}
	\end{figure}

	Therefore,
	\begin{equation}\label{A}
		\int_{\gamma_m}\frac{|dz|}{\delta_{\Omega}(z)}=\int_{\gamma_m^1}\frac{|dz|}{\delta_{\Omega}(z)}+\int_{\gamma_m^2}\frac{|dz|}{\delta_{\Omega}(z)}+\int_{\gamma_m^3}\frac{|dz|}{\delta_{\Omega}(z)}
	\end{equation}
	If $z\in\gamma_m^3$, then $\delta_{\Omega}(z)\le r_0$ and $\int_{\gamma}|dz|\ge r_0$, therefore 
	\begin{equation}\label{a}
		\int_{\gamma_m^3}\frac{|dz|}{\delta_{\Omega}(z)}\ge 1.
	\end{equation}
	If $z\in\gamma_m^2$, then $\delta_{\Omega}(z)\le r_m^{\alpha}$ and $\int_{\gamma}|dz|\ge r_m^{\beta}$, and hence
	\begin{equation}\label{b}
		\int_{\gamma_m^2}\frac{|dz|}{\delta_{\Omega}(z)}\ge r_m^{\beta-\alpha}.
	\end{equation}
	If $z\in\gamma_m^1$, then $\delta_{\Omega}(z)\le r_m$ and $\int_{\gamma}|dz|\ge r_m$, and hence
	\begin{equation}\label{c}
		\int_{\gamma_m^1}\frac{|dz|}{\delta_{\Omega}(z)}\ge 1.
	\end{equation}
	By using \eqref{a}, \eqref{b} and \eqref{c}, in \eqref{A}, we obtain that the quasihyperbolic diatance from $x_0$ to $x_m$ is
	\begin{equation}\label{key}
		k_{\Omega}(x_0,x_m)\ge \left(\frac{1}{\delta_{\Omega}(x_m)}\right)^{\alpha-\beta}+2.
	\end{equation}
	In view of \eqref{EQHBC} and \eqref{key}, we have
	$$\left(\frac{1}{\delta_{\Omega}(x_m)}\right)^{\alpha-\beta}\le \frac{1}{B} \frac{1}{\delta_{\Omega}(x_m)}+C,$$
	for some $C$. This is a contradiction to the fact that $\alpha-\beta>1.$ Hence $\Omega$ is not an QHBC domain. 
	
	In view of Theorem \ref{LC boundary}, we conclude that $\Omega$ is a visibility domain. However, in view of Claim 1 and claim 2 $\Omega$ is neither a John nor QHBC domain.	
	
	\begin{example}[\textbf{Visibility domains which are not Gromov hyperbolic}]
	Since we know that every John domain is a visibility domain, therefore to construct domains of this type the idea is to construct John domains in which there exists a quasihyperbolic geodesic which is not a cone arc, and hence such domains can not be Gromov hyperbolic (see \cite{Rasila-2022}). For examples of such domains we refer to  \cite[Example 5.12 and Remark 5.17]{Gehring-1989}.
	\end{example}

	\begin{example}[\textbf{Gromov hyperbolic domains which are not visibility domains}]
		To construct domains of this class one can use Theorem \ref{Main-thm-GB-EB}. More precisely, the idea is to construct Gromov hyperbolic domains such that the identity map does not extends as a continuous surjective map from Gromov closure to Euclidean closure. Moreover, as we know that every planar simply connected domain is Gromov hyperbolic but it is a visibility domain if, and only if, its boundary is locally connected.
	\end{example}

\section{Appendix}\label{Appendix}
	%	\subsection{Comparison between visibility and visibility in Gromov hyperbolic spaces}
	In this section, we first prove Theorem \ref{GV} and then we see the comparison of visibility between Hadamard manifolds and CAT($0$) spaces for the sake of completeness.
	
		\begin{pf}[Proof of Theorem \ref{GV}]
	To prove Theorem \ref{GV} it is enough to prove Remark \ref{GV}. Suppose $X$ is a Gromov hyperbolic space with some $\delta$. Let $\xi,\eta\in \partial_{G}X$, $\xi\ne \eta$ and $\{z_n\}, \{w_n\} \subset X$ be sequences such that $z_n \stackrel{Gromov}\longrightarrow \xi$ and $w_n \stackrel{Gromov}\longrightarrow \eta$. Consider any sequence of geodesics $\Gamma_n$ joining $z_n$ and $w_n$. Fix a point $p\in X$, %we can choose representativies $\gamma$ and $\sigma$ of $\xi$ and $\eta$ respectively such that there exists a sequence $t_n\to \infty$ with $\gamma(t_n)=z_n$ and $\sigma(t_n)=w_n$. 
	and let $\gamma_n$ and $\sigma_n$ be sequences of geodesics joining $p$ and $z_n$, $p$ and $w_n$ respectively. As $z_n \stackrel{Gromov}\longrightarrow \xi$, every subsequence of $\gamma_n$ has a subsequence converges on compacta to a geodesic ray $\gamma\in \xi$. Similarly, we get a geodesic ray $\sigma\in \eta$. Since both $\gamma$ and $\sigma$ are not equivalent, therefore we can choose a $T$ such that 
	$$d(\gamma(T), \mbox{range}(\sigma))>\delta.$$
	We can also choose $T$ such that
\begin{equation}\label{App1}
	\forall R> T,\,\, d(\gamma(R),\mbox{range}(\sigma))>\delta
\end{equation}
%	the distance between $\gamma(T)$ and image of $\sigma$ is greater than $\delta$.
	\par For sufficiently large $n$, we can consider $\Gamma_n$ as a sequence of geodesics joining points of geodesic rays $\gamma$ and $\sigma$. For $R>T$ consider the geodesic triangle with sides $\gamma|_{[0,R]}$, $\Gamma_n|_{[\gamma(R),\sigma(R)]}$ and $\sigma|_{[0,R]}$. Since $(X,d)$ is Gromov hyperbolic, we have that  this triangle is $\delta$-thin for some $\delta<\infty$ {\it i.e.,}
	\begin{equation}\label{App2}
	 \forall t\le R,\, \min\{d(\gamma(t),\sigma([0,R])), d(\gamma(t),\Gamma_n([\gamma(R),\sigma(R)]))\}\le \delta
	\end{equation}
	
		Using \eqref{App1} and \eqref{App2}, we obtain
	$$\forall R>T,\,\, d(\gamma(T),\Gamma_n([\gamma(R),\sigma(R)]))\le\delta$$
	Consider the following compact set $K$ of $\Omega$
	$$K:=\{z\in \Omega:d(z,\gamma(T))\le \delta\}.$$
	Therefore, geodesics $[\Gamma_n(R), \Gamma_n(R)]$ must intersect the compact set $K$. This completes the proof.
\end{pf}

We now recall the visibility in Hadamard manifolds and see its equivalent formulation. Further, we see its generalization to CAT($0$) spaces. 

\subsection{Visibility in Hadamard manifolds}
A Hadamard manifold $H$ is a complete, simply-connected Riemannian manifold with non-positive sectional curvature $K\le 0$, and a complete Riemannian manifold $M$ of dimension $n\ge 2$ with non-positive sectional curvature are precisely the quotient manifolds $M=H/D$, where $D$ is group of isometries of $H$ which acts properly discontinuously on $H$. Eberlein and O'Neill \cite{O'Neill-1973} have introduced a general construction of compactification of Hadamrd manifolds by first attaching a boundary at infinity, which is they called the ideal boundary, and then defining the natural topology, namely the cone topology. For Hadamard manifolds this compactification is homeomorphic to closed $n$-ball. The notion of ideal boundary is a central concept in the theory of negative or non-positively curved metric spaces and this notion comes after defining a suitable notion of negative or non-positive curvature. This notion also has a generalization in CAT($0$) spaces and Gromov hyperbolic spaces respectively.
\par Eberlein and O' Neill have given the following definitions (see \cite[Definition 4.1]{O'Neill-1973}) 
\begin{defn}[Axiom 1]
	We say that a Hadamard manifold $H$ satisfies Axiom 1, if for any points $x\ne y$ in $H(\infty)$ there exists at least one geodesic line joining $x$ and $y$.
\end{defn}
\begin{defn}[Visibility manifold]\label{Neil-def}
	Let $M=H/D$ be a complete Riemannian manifold with non-positive sectional curvature. If $H$ satisfies Axiom 1, we call $M$ a {\it visibility manifold}.
\end{defn}
It is well-known (see \cite[Lemma 9.10]{O'Neill-1969} or \cite{O'Neill-1973}) that if $H$ is complete, simply-connected Riemannian manifold with sectional curvature $K\le c<0$, then $M=H/D$ is a visibility manifold. In particular, every complete, simply-connected Riemannian manifold with constant negative curvature is a visibility manifold. In fact, visibility is the natural replacement for the strict negative curvature. Moreover, Eberlein and O'Neill \cite{O'Neill-1973} also defined a property called {\it visibility axiom} which is equivalent to Axiom 1 (see \cite[Definition 4.2]{O'Neill-1973}). For a Hadamard manifold $H$, following theorem gives the equivalence of different visibility criteria.
\begin{thm}\cite[Proposition 4.4]{O'Neill-1973}\label{Visibility in HM}
	Let $H$ be a Hadamard manifold then following are equivalent
	\begin{itemize}
		\item[(1)] $H$ satisfies Axiom 1
		\item[(2)] For each pair of distinct points $\xi, \eta\in H(\infty)$ and for any sequence $z_k\to \xi$ and $w_k\to \eta$ there exist a compact set $K\subset H$ such that:
		geodesics $\gamma_k$ in $H$ joining $z_k$ and $w_k$ intersects with $K$.
	\end{itemize}
\end{thm}
The notion of visibility has also its generalization in CAT($0$) spaces. In CAT(0) spaces the condition of Axiom 1 is known as visibility space and the condition of visibility axiom is known as locally visibile space. As we have seen that in the case of Hadamard manifolds they are equivalent, but this is not the case in CAT($0$) spaces (see \cite[Remarks 9.31]{Bridson-book}). Nevertheless, for proper CAT($0$) spaces they are equivalent. Moreover, for a Hadamard manifolds there are different equiavlent ways in which visibility can be expressed (see \cite[p. 54]{Gromov-book}), and most of these different equivalent ways of visibility can be generalized to CAT($0$) spaces (see \cite[Chapter II.9, Proposition 9.35]{Bridson-book}).

\begin{rem}\label{Visibility in quasihyperbolic geometry}
At this stage, one may be curious about: what if we consider Definition \ref{Neil-def} to study the visibility property in the context of quasihyperbolic geometry? But for this we need the notion of ideal boundary of the space $(\Omega,k_{\Omega})$. For the notion of ideal boundary we need a suitable notion of negative or nonpositive curvature. For this we have two choices one is Gromov hyperbolicity and the other is CAT(0). 
\par If we consider Gromov hyperbolicity, then we have the complete answer since every Gromov hyperbolic space satisfies the visibility property in the sense of Eberlein and O'Neill ({\it i.e. connecting ideal boundary points by geodesic line}) by considering ideal boundary as Gromov boundary.
\par If we consider the notion of CAT($0$) space, then we can study visibility in the sense of Eberlein and O'Neill. For this we need to know when $(\Omega,k_{\Omega})$ is CAT($0$). But this is going to be restrictive because of the following two facts:
\begin{itemize}
	\item[(i)] By Cartan-Hadamard theorem we know that a complete metric space is CAT($0$) if and only if it is locally CAT($0$) and simply connected and,
	\item[(ii)] the quasihyperbolic metric induces the same topology and $(\Omega,k_{\Omega})$ is a complete metric space.
\end{itemize}	
Thus, in view of the above two facts, {\it we can only consider simply connected domains in $\mathbb{R}^n$ to study visibility in the sense of Eberlein and O'Neill.}
\end{rem}

	\section{Concluding remarks and topics for further exploration}
	We conclude this paper with the following remark 
		\begin{rem}
		One can show that if $p,q\in \partial_{Euc}\Omega$, $p\ne q$, has visible qh-geodesics then 
		$$\limsup_{k\to \infty}(z_k|w_k)_o< \infty,$$ 
		where, $z_k\to p$, $w_k\to q$. At present, we do not know whether the converse holds or not? Moreover, in the context of visibility with respect to Kobayashi distance, Bracci {\it et. al.}\cite[Proposition 4.3]{Bracci-2022} have constructed a counter example for bounded convex domains in $\mathbb{C}^n$. But every bounded convex domain $\mathbb{R}^n$ is a QH-visibility domain. Thus, in the context of visibility with respect to quasihyperbolic metric it would be interesting to check whether the converse holds or not.
	\end{rem}	
	At present, we can think of following topics for further exploration on visibility in quasihyperbolic geometry.
	\subsection*{Topic 1: H-visibility and QH- visibility} At present, we know of no hyperbolic domain, which is a visibility domain with respect to one metric but not the other. The following problem seems challenging.
	
\begin{prob}\label{Main-Pb}
Does there exist a hyperbolic planar domain that is a visibility domain with respect to one metric but not the other? If such domain exists, the next problem is to describe this class of domains.
\end{prob}

\subsection*{Topic 2: Visibility and CAT($0$) geometry of quasihyperbolic metric} In view of Remark \ref{Visibility in quasihyperbolic geometry}, we can study visibility in the sense of Eberlein and O'Neill (see Definition \ref{Neil-def}) only for simply connected domains of $\Omega\subset\mathbb{R}^n$. Moreover, in the case of planar domains recently Herron \cite[Theorem A]{Herron-2021} has proved that every simply connected domain with quasihyperbolic metric is CAT($0$). Thus, if we consider planar simply connected domains, in view of the following two facts:
\begin{itemize}
	\item[(i)] If $\Omega$ is simply connected planar domain, then $(\Omega,k_{\Omega})$ is Gromov hyperbolic and hence its ideal boundary is Gromov boundary.
	\item[(ii)] Every Gromov hyperbolic space has visibility property when considering the Gromov boundary in the Gromov topology (see \cite[Part III, H, Lemma 3.2]{Bridson-book}),
\end{itemize}
we have that $(\Omega,k_{\Omega})$ always has the visibility property in the sense of Eberlein and O'Neill (see Definition \ref{Neil-def}), but $\Omega$  is a visibility domain in the sense of Definition \ref{main-defn} if, and only if, $\partial_{Euc}\Omega$ is locally connected. Since in dimension greater than and equal to $3$ not every simply connected domain is Gromov hyperbolic, therefore studying the visibility of bounded simply connected domains in $\mathbb{R}^n$, $n\ge 3$ in the sense of Eberlein and O'Neill could be an interesting task. For this we first need to consider the following problem:
\begin{prob}\label{CAT(0)}
	Let $\Omega\subset \mathbb{R}^n$, $n>2$ be a bounded simply connected domain. When is $(\Omega,k_{\Omega})$ CAT($0$) space?
\end{prob}
Then one can pose the following problem
\begin{prob}
	When a bounded simply connected domain $\Omega\subset \mathbb{R}^n$ has the visibility property in the sense of Eberlein and O'Neill?
\end{prob}

%\subsection*{title} Let $\Omega$ be a proper subdomain in $\mathbb{R}^2$, and $\rho$ be a $C^2$ smooth metric. The Gaussian curvature of this metric at $z_0\in \Omega$ is defined as
%\begin{equation}\label{GC}
%	\mathcal{K}_{\rho}(z_0)=-\frac{\Delta \log\,\rho(z_0)}{(\rho(z_0))^2}
%\end{equation}
%Heins in \cite{Heins-1962}, defined the Gaussian curvature for more general metric (which need not to be $C^2$). Using the integral formula for Laplacian, Heins in \cite{Heins-1962}, defined the upper Gaussian curvature $\overline{\mathcal{K}}_{\rho}$ and lower Gaussian curvature $\underline{\mathcal{K}}_{\rho}$ of the conformal metric $\rho(z)|dz|$. We denote the upper Gaussian, lower Gaussian and Gaussian curvature by $\overline{\mathcal{K}}_{\Omega}$, $\underline{\mathcal{K}}_{\Omega}$ and $\mathcal{K}_{\Omega}$ respectively to indicate the dependence on $\Omega$ equipped with quasihyperbolic metric. Upper and lower Gaussian curvature for quasihyperbolic metric always exists but it may be $-\infty$. Martin and Osgood \cite{Martin-1986} studied this notion of curvature, in the sense of Heins, for Planar domains.\\
%Martin and Osgood \cite[Corollary 3.6]{Martin-1985} have proved that for any Planar domain the upper generalized Gaussian curvature for quasihyperbolic metric is non-positive. Thus, It would be interesting if one can relate the generalized Gaussian curvature of a domanin $\Omega$ with the $(\Omega,k_{\Omega})$ being locally CAT($0$).

	\subsection*{Topic 3: QH-visibility of Poincar\'e domains} Following \cite{Koskela-2002}, a finite volume ($n$-dimensional Lebesgue measure is finite) domain $\Omega\subset \mathbb{R}^n$, $n\ge 2$ is said to be $(q,p)$-Poincar\'e domain, where $1\le p\le q <\infty$, if there exists a constant $M_{p,q}=M_{p,q}(\Omega)$ such that
	\begin{equation}\label{Poincare}
		\left(\int_{\Omega}|u(x)-u_{\Omega}|^q\,dx\right)^{1/q}\le M_{p,q}\left(\int_{\Omega}|\nabla u(x)|^{p}\,dx\right)^{1/p}
	\end{equation}
	for all $u\in C^{\infty}(\Omega)$, where 
	$$u_{\Omega}=|\Omega|^{-1}\int_{\Omega}u(x)\,dx.$$
	When $q=p$, we say that $\Omega$ is a $p$-Poincar\'e domain. Determining a geometric condition on a domain $\Omega$ (which possibly have quite irregular boundary) which is sufficient for a domain to be Poincar\'e domain is an interesting problem. In \cite{Koskela-2002}, Koskela {\it et. al.} considered one such condition as a quasihyperbolic boundary condition and proved a condition on $p$ so that $\Omega$ is a $p$-Poincar\'e domain. By the general visibility criteria we know that the domain with quasihyperbolic boundary condition is a visibility domain. Thus, it seems plausible if we can use the visibility property to study Poincar\'e domains.
	
	\subsection*{Topic 5: Dynamics of Quasiregular maps} The concept of visibility in the context of Kobayashi distance has been succesfully used to prove some Wolf-Denjoy type results for holomorphic as well as some non-holomorphic functions (see \cite[Theorem 1.10 and Theorem 1.11]{Bharali-2017}, \cite[Theorem 1.8 and Theorem 1.9]{Bharali-2021}, \cite[Theorem 1.15]{Sarkar-2021} and \cite[Theorem 1.13]{Bharali-2022}). Recently, Fletcher and Macclure \cite[Section 7]{Fletcher-2020} have established Wolf-Denjoy type results for quasiregular mappings. Thus, the question is that can we somehow use visibility in context of the iteration theory of quasiregular maps?
	
	%\subsection*{Acknowledgements}
	%\addtocontents{toc}{\protect\setcounter{tocdepth}{1}}
	\section{Acknowledgements}
	We are greatful to Prof. Pekka Koskela for introducing us the paper \cite[Theorem 1.1]{Lammi-2011} in the context of Problem \ref{GB-EB} in the present paper during the research visit of the second named author at the department of mathematics and statistics, Jyv\"askyl\"a university. This helped us to significantly modify the introduction of the manuscript and establishing the Corollary \ref{Quasiconvex-no loop}. Authors would like to thank Prof. Gautam Bharali for series of interesting talks on visibility in the context of Kobayashi metric at the SCV conference at KSOM, Kozhikode, India. We also thank Dr. Amar Deep Sarkar for fruitful discussions on visibility. The first named author thanks SERB-CRG, and the second named author thanks PMRF-MHRD (Id: 1200297), Govt. of India for their support.

\end{document}